\documentclass{amsart}
\usepackage{amsmath,amssymb, amsbsy}
\usepackage[dvips]{graphicx}
\usepackage{hyperref}
\usepackage{enumerate}
\renewcommand{\Re}{\mathfrak{Re}}
\newcommand{\R}{{\mathbb R}}
\newcommand{\N}{{\mathbb N}}
\newcommand{\Z}{{\mathbb Z}}
\newcommand{\C}{{\mathbb C}}

\newcommand{\inn}{\text{  in   }}
\newcommand{\dyle}{\displaystyle}
\newcommand{\dint}{\dyle\int}
\newcommand{\weakly}{\rightharpoonup}
\newcommand{\e }{\varepsilon}

\newcommand{\dive }{\mathop{\rm div}}

\newcommand{\A}{{\mathbf A}}

\renewcommand{\geq }{\geqslant}

\renewcommand{\leq }{\leqslant}

\newenvironment{pf}{\noindent{\sc Proof}.\enspace}{\hfill\qed\medskip}
\newenvironment{pfn}[1]{\noindent{\bf Proof of
    {#1}.\enspace}}{\hfill\qed\medskip}
\newtheorem{Theorem}{Theorem}[section]
\newtheorem{Corollary}[Theorem]{Corollary}
\newtheorem{Lemma}[Theorem]{Lemma}
\newtheorem{Proposition}[Theorem]{Proposition}
\theoremstyle{definition} \newtheorem{Definition}[Theorem]{Definition}

\newtheorem{remark}[Theorem]{Remark}
\begin{document}

\title[Parabolic equations with critical
electromagnetic potentials]
{On parabolic equations with critical electromagnetic
potentials}

\author[Veronica Felli]{Veronica Felli}
\address{\hbox{\parbox{5.7in}{\medskip\noindent{Universit\`a di Milano--Bicocca,\\
        Dipartimento di Scienza dei Materiali, \\
        Via Cozzi
        55, 20125 Milano, Italy. \\[3pt]
        \em{E-mail address: }{\tt veronica.felli@unimib.it}.}}}}
         
\author[Ana Primo]{Ana Primo}       
\address{ \hbox{\parbox{5.7in}{\medskip\noindent{Universidad  Aut\'onoma de  Madrid,\\ Departamento de  Matem{\'a}ticas, \\  28049 Madrid, Espa\~na.}\\[3pt]
 \em{E-mail address: }{\tt ana.primo@uam.es}.}}}
  
\date{October 22, 2018}

\thanks{
  V. Felli is partially
supported by the PRIN2015 grant ``Variational methods, with
applications to problems in mathematical physics and
geometry''. A. Primo is supported by project MTM2016-80474-P, MEC, Spain}

 \begin{abstract}
   \noindent  
We consider a class of parabolic equations with critical electromagnetic
potentials, for which we obtain a classification of local asymptotics,
unique continuation results,
and an integral representation formula for solutions. \end{abstract}

\maketitle

\section{Introduction and statement of the main results}\label{intro}

This paper is concerned with the following class of  evolution equations
with critical electromagnetic potentials
\begin{equation}\label{prob}
u_t+\left(-i\nabla+ \dfrac{{\mathbf{A}}\big(\frac{x}{|x|}\big)} {|x|}
\right)^{\!\!2} u- \dfrac{a\big(\frac{x}{|x|}\big)}{|x|^2}\,u= h(x,t) u,
\end{equation}
in $\R^N\times I$, for some interval $I\subset\R$ and for $N\geq 2$. 
Here $u=u(x,t):{\mathbb{R}}^{N}\times I\to{\mathbb{C}}$,  $a\in
L^{\infty}({\mathbb{S}}^{N-1}, {\mathbb{R}})$, ${\mathbb{S}}^{N-1}$ denotes
the unit $(N-1)$--dimensional sphere, and ${\mathbf{A}}\in C^1({\mathbb{S}} ^{N-1},{\mathbb{R}}^N)$ satisfies the following  transversality condition
\begin{equation}  \label{transversality}
{\mathbf{A}}(\theta)\cdot\theta=0 \quad \text{for all }\theta\in {\mathbb{S}}
^{N-1}.
\end{equation}
We always denote by $r:=|x|$, $\theta=x/|x|$, so that $x=r\theta$.
Under the transversality condition \eqref{transversality}, the hamiltonian
\begin{equation}  \label{eq:hamiltonian}
{\mathcal{L}}_{{\mathbf{A}},a}:= \left(-i\,\nabla+\frac{{\mathbf{A}}\big(%
\frac{x}{|x|}\big)} {|x|}\right)^{\!\!2}-\dfrac{a\big(\frac{x}{|x|}\big)}{%
|x|^2}
\end{equation}
formally acts on functions $f:{\mathbb{R}}^N\to{\mathbb{C}}$ as
\begin{equation*}
{\mathcal{L}}_{{\mathbf{A}},a}f= -\Delta  f+\frac{|{\mathbf{A}}\big(\frac{x}{%
|x|}\big)|^2- a\big(\frac{x}{|x|}\big)-i\dive\nolimits_{{\mathbb{S}}^{N-1}}{%
\mathbf{A}}\big(\frac{x}{|x|}\big)
  }{|x|^2}\,f-2i\,\frac{{\mathbf{A}}\big(\frac{x}{|x|}\big)}{|x|}\cdot\nabla
f,
\end{equation*}
where $\dive\nolimits_{{\mathbb{S}}^{N-1}}{\mathbf{A}}$ denotes the
Riemannian  divergence of ${\mathbf{A}}$ on the unit sphere ${\mathbb{S}}%
^{N-1}$ endowed  with the standard metric.

The electromagnetic potential appearing in \eqref{prob} is singular
and homogeneous.   A prototype in dimension $2$ of such type of
potentials is given by the   Aharonov-Bohm vector potential 
\begin{equation}\label{eq:72}
(x_1,x_2)\mapsto {\boldsymbol{\mathcal A}}(x_1,x_2)=
\alpha\left(-\frac{x_2}{x_1^2+x_2^2},\frac{x_1}{x_1^2+x_2^2}\right)
\end{equation}
which is associated to thin solenoids. If the radius of the solenoid
tends to zero while the flux through it remains constantly equal to
$\alpha\not\in\Z$, then a $\delta$-type magnetic field is produced and the so-called
 Aharonov-Bohm effect occurs, i.e. the magnetic potential affects
 charged quantum particles
 moving in $\R^2\setminus\{(0,0)\}$, even if the magnetic field
 ${\boldsymbol{\mathcal B}}=
 \mathop{\rm curl}{\boldsymbol{\mathcal A}}$ is zero there.

We mention that heat semigroups generated by  magnetic
Schr\"odinger operators have been studied in \cite{cazacu,kovarik}.
In particular in \cite{cazacu} the case of a compactly supported smooth
magnetic field was considered and it was shown that the large time
behavior of the heat semigroup is related to a magnetic  eigenvalue problem 
on the $(N-1)$-dimensional sphere (see \eqref{angular}). In \cite{kovarik}
it was proved that the large time behavior of magnetic heat kernels in two dimensions is
determined by the flux of the
magnetic field.

In the present paper, we aim at providing some unique continuation principles
for problem \eqref{prob} under suitable assumptions on the
perturbing potential $h$ and deriving a representation formula for
solutions in the case $h\equiv0$.

In order to establish unique continuation properties, we will describe
the asymptotic behaviour of solutions near the singularity, under the
assumption that, in some bounded interval $I$,   the real-valued
function $h$ satisfies
\begin{equation}\label{eq:der}
\begin{cases}
h,h_t\in L^{r}\big(I,L^{{N}/{2}}(\R^N)\big)
\quad\text{for some }r>1,
\quad h_t\in L^{\infty}_{\rm
  loc}\big(I,L^{{N}/{2}}(\R^N)\big),&\text{if }N\geq3,\\
h,h_t\in L^{r}\big(I,L^{p}(\R^N)\big)
\quad\text{for some }p,r>1,
\quad h_t\in L^{\infty}_{\rm
  loc}\big(I,L^{p}(\R^N)\big),&\text{if }N=2,\\
\end{cases}
\end{equation}
and there exists $C_h>0$ such that
\begin{equation}\label{eq:h}
  |h(x,t)|\leq C_h(1+|x|^{-2+\e}) \quad
  \text{for all }t\in I,\text{ a.e. }x\in\R^N,
  \text{ and for some }\e\in(0,2).
\end{equation}
In particular,  for $t_0\in I$ fixed, we are interested in describing the
behavior of solutions along the directions $(\lambda x,t_0-\lambda^2
t)$ 
naturally related to the heat operator. Indeed, since the unperturbed
operator $u_t+\left(-i\nabla+ \frac{{\mathbf{A}}({x}/{|x|})} {|x|}
\right)^{\!2} u- \frac{a({x}/{|x|})u}{|x|^2}\,$ is invariant under
the action $(x,t)\mapsto(\lambda x,t_0+\lambda^2 t)$, we are
interested in evaluating the asymptotics of
$$
u(\sqrt{t_0-t}\, x,t)\quad\text{as }t\to t_0^-
$$
for solutions to (\ref{prob}).  We notice that, in the magnetic-free
case $\A\equiv 0$, 
an asymptotic analysis for solutions to (\ref{prob}) was developed in \cite{FP}.

In the  description of the asymptotic behavior at the singularity of
solutions to (\ref{prob}) a key role is played by 
the eigenvalues and eigenfunctions of the Ornstein-Uhlenbeck magnetic operator
with singular inverse square potential
\begin{equation}\label{eq:13}
L_{\A,a}:{\mathcal H}\to{\mathcal H}^\star,\quad
L_{\A,a}= {\mathcal{L}}_{{\mathbf{A}},a}+\frac{x}2\cdot\nabla,
\end{equation}
defined as
$$
{}_{{\mathcal H}^\star}\langle L_{\A,a}v,w \rangle_{{\mathcal H}} =
\int_{\R^N}\bigg(\nabla_\A v(x)\cdot\overline{\nabla_\A
  w(x)}-\frac{a(x/|x|)}{|x|^2}\,v(x)\overline{w(x)}\bigg)e^{-\frac{|x|^2}{4}}\,dx,
\quad\text{for all }v,w\in{\mathcal H}.
$$
Here $\mathcal H$ is the functional space defined as 
the completion of $C^{\infty}_{\rm c}(\R^N\setminus\{0\},\C)$ with
respect to the norm
\begin{equation}\label{eq:Hstar}
\|v\|_{\mathcal{H} }=\bigg(\int_{{\mathbb{R}}^N}\bigg(|\nabla v(x)|^2+
|v(x)|^2+\frac{|v(x)|^2}{|x|^2}\bigg) e^{-\frac{|x|^2}4} \,dx\bigg)^{\!\!1/2}.
\end{equation}
The spectrum of $ {{L}}_{{\mathbf{A}},a}$ is related to the angular component
of the operator ${\mathcal{L}}_{{\mathbf{A}},a}$ on the unit 
$(N-1)$--dimensional sphere ${\mathbb{S}}^{N-1}$, i.e. to the operator
\begin{align}\label{eq:angular}
T_{{\mathbf{A}},a} & =\big(-i\,\nabla_{\mathbb{S}^{N-1}}+{\mathbf{A}}\big)%
^2-a(\theta) =-\Delta_{\mathbb{S}^{N-1}}+\big(|{\mathbf{A}}|^2-a(\theta)-i\,\dive%
\nolimits_{{\mathbb{S}}^{N-1}}{\mathbf{A}}\big)-2i\,{\mathbf{A}}\cdot\nabla_{%
\mathbb{S}^{N-1}}.  \notag
\end{align}
By classical spectral theory, $T_{{\mathbf{A}},a}$ admits a diverging
sequence of real eigenvalues with finite multiplicity $\mu_1({\mathbf{A}}%
,a)\leq\mu_2({\mathbf{A}},a)\leq\cdots\leq\mu_k({\mathbf{A}},a)\leq\cdots$,
see \cite[Lemma A.5]{FFT}.
The   first eigenvalue $\mu_1(\A,a)$ admits the following variational characterization:
\begin{equation}\label{firsteig}
  \mu_1(\A,a)=\min_{\psi\in H^1(\mathbb
    S^{N-1})\setminus\{0\}}\frac{\int_{\mathbb S^{N-1}}
    \big[\big|\big(\nabla_{\mathbb S^{N-1}}+i\A(\theta)\big)\psi(\theta)\big|^2
    -a(\theta)|\psi(\theta)|^2\big]\,dS(\theta)}{\int_{\mathbb S^{N-1}}
    |\psi(\theta)|^2\,dS(\theta)}.
\end{equation}
 To each $k\in{\mathbb{N}}$, $k\geq 1$, we
associate a $L^{2}\big({\mathbb{S}}^{N-1},{\mathbb{C}}\big)$-normalized
eigenfunction $\psi_k$ of the operator $T_{{\mathbf{A}},a}$ on $\mathbb{S}%
^{N-1}$ corresponding to the $k$-th eigenvalue $\mu_{k}({\mathbf{A}},a)$,
i.e. satisfying
\begin{equation}  \label{angular}
\begin{cases}
T_{{\mathbf{A}},a}\psi_{k}=\mu_k({\mathbf{A}},a)\,\psi_k(\theta), & \text{in
}{\mathbb{S}}^{N-1}, \\[3pt]
\int_{{\mathbb{S}}^{N-1}}|\psi_k(\theta)|^2\,dS(\theta)=1. &
\end{cases}%
\end{equation}
In the enumeration $\mu_1({\mathbf{A}},a)\leq\mu_2({\mathbf{A}}%
,a)\leq\cdots\leq\mu_k({\mathbf{A}},a)\leq \cdots$ we repeat each eigenvalue
as many times as its multiplicity, so that exactly one eigenfunction $\psi_k$
corresponds to each index $k\in{\mathbb{N}}$. Furthermore, the
functions $\psi_k$
 can be chosen  in such a way that they form an orthonormal basis of $L^2({\mathbb{S}%
}^{N-1},{\mathbb{C}})$. 
We mention that the key role played by the  angular magnetic Schr\"odinger
operator  $T_{\A,a}$
in 
the  behaviour of the heat magnetic semigroup was 
already highlighted  in \cite{cazacu}.

We notice that, under the condition
\begin{equation}  \label{eq:hardycondition}
\mu_1({\mathbf{A}},a)>-\left(\frac{N-2}{2}\right)^{\!\!2}
\end{equation}
the quadratic form associated to $\mathcal{L}_{{\mathbf{A}},a}$ is positive
definite (see \cite[Lemma 2.2]{FFT}), thus implying that the hamiltonian $\mathcal{L}_{{\mathbf{A}},a}$ is
a symmetric semi-bounded operator on $L^2({\mathbb{R}}^N;{\mathbb{C}})$
which then admits a self-adjoint extension (Friedrichs extension) with the
natural form domain.

We introduce the notation
\begin{equation}  \label{eq:alfabeta}
\alpha_k:=\frac{N-2}{2}-\sqrt{\bigg(\frac{N-2}{2}\bigg)^{\!\!2}+\mu_k({%
\mathbf{A}},a)}, \quad \beta_k:=\sqrt{\left(\frac{N-2}{2}\right)^{\!\!2}+%
\mu_k({\mathbf{A}},a)},
\end{equation}
so that $\beta_{k}=\frac{N-2}{2}-\alpha_{k}$, for all $k\in\N$, $k\geq1$.

The first result of the present paper is the following complete description of the
spectrum of the operator $L_{\A,a}$. We mention that analogous results
were proved in the case where $a$ is a constant and $\A\equiv0$ in \cite[\S
9.3]{vazquez_zuazua} and in the magnetic-free case $\A\equiv0$ in
\cite[Proposition 1]{FP}; see also \cite[\S 4.2]{CH} and \cite[\S
2]{ES} for the non singular case.

\begin{Proposition}\label{p:explicit_spectrum}
The set of the eigenvalues  of the operator $L_{\A,a}$ is
$\big\{ \gamma_{m,k}: k,m\in\N, k\geq 1\big\}$
where
\begin{equation}\label{eq:65}
\gamma_{m,k}=m-\frac{\alpha_k}2, 
\end{equation}
being $\alpha_k$ as in \eqref{eq:alfabeta}. Each eigenvalue
$\gamma_{m,k}$ has finite multiplicity equal 
to
$$
\#\bigg\{j\in\N,j\geq 1:
\gamma_{m,k}+\frac{\alpha_j}2\in\N\bigg\}
$$
and a basis of the corresponding eigenspace is
$\left\{V_{n,j}: j,n\in\N,j\geq
  1,\gamma_{m,k}=n-\frac{\alpha_j}2 \right\}$,
where
\begin{equation}\label{eq:66}
V_{n,j}(x)=
|x|^{-\alpha_j}P_{j,n}\bigg(\frac{|x|^2}{4}\bigg)
\psi_j\Big(\frac{x}{|x|}\Big),
\end{equation}
$\psi_j$ is an eigenfunction of the operator
$T_{A,a}$ 
associated to the $j$-th eigenvalue $\mu_{j}(\A,a)$ as in \eqref{angular},
  and $P_{j,n}$
is the polynomial of degree $n$ given by
$$
P_{j,n}(t)=\sum_{i=0}^n
\frac{(-n)_i}{\big(\frac{N}2-\alpha_j\big)_i}\,\frac{t^i}{i!},
$$
denoting as $(s)_i$, for all $s\in\R$, the Pochhammer's symbol
$(s)_i=\prod_{j=0}^{i-1}(s+j)$, $(s)_0=1$.
\end{Proposition}

The second main result of the present paper establishes
a sharp relation between the asymptotic behaviour of solutions
to (\ref{prob}) along the directions $(\lambda x,t_0-\lambda^2
t)$ and the spectrum of the operator $L_{\A,a}$.
Indeed we prove that 
\[
u(\sqrt{t_0-t}\, x,t)\sim (t_0-t)^\gamma
g(x)\quad\text{as as $t\to
t_0^-$},
\]
 where $\gamma$ is an eigenvalue of $L_{\A,a}$
and $g$ is an associated eigenfunction. In order to state precisely
the result of our asymptotic analysis, we introduce 
the Hilbert space  ${\mathcal{H}_t }$
 defined, for every $t>0$,  as the completion of 
$C^{\infty}_{\mathrm{c}}({\mathbb{R}}^N\setminus\{0\},{\mathbb{C}})$ with respect to the norm
\begin{equation}\label{eq:star}
\|\phi\|_{\mathcal{H}_{t} }=\bigg(\int_{{\mathbb{R}}^N}\bigg(t|\nabla\phi(x)|^2+
|\phi(x)|^2+t\frac{|\phi(x)|^2}{|x|^2}\bigg)G(x,t) \,dx\bigg)^{\!\!1/2},
\end{equation}
where 
\begin{equation*}
G(x,t)=t^{-N/2}\exp\Big(-\frac{|x|^2}{4t}\Big)
\end{equation*}
is the heat kernel of the free evolution forward equation satisfying
\begin{equation}\label{eq:heatker}
 G_t-\Delta G=0\quad\text{and}\quad \nabla G(x,t)=-\frac{x}{2t}\,G(x,t)
\quad\text{in }\R^N\times (0,+\infty).
\end{equation}
For every $t>0$, we also define  the space ${\mathcal L}_t$ as
the completion of $C^{\infty}_{\rm c}(\R^N)$ with respect to the norm
$$
\|u\|_{{\mathcal L}_t}=\bigg(\int_{\R^N}|u(x)|^2 G(x,t)\,dx\bigg)^{\!\!1/2}.
$$
We set $\mathcal L:=\mathcal L_1$.
\begin{Theorem}\label{asym1}
Let
$N\geq2$,  ${\mathbf{A}}\in C^1({\mathbb{S}}
^{N-1},{\mathbb{R}}^N)$ and $a\in L^{\infty}\big({\mathbb S}^{N-1},\R\big)$ satisfy \eqref{transversality}
and
  \eqref{eq:hardycondition}. 
  Let $u\not\equiv 0$ be a weak solution to (\ref{prob}) (see 
  Definition \ref{def:solution} for the notion of weak solution)
in $\R^N\times(t_0-T,t_0)$ with  $h$ satisfying
  \eqref{eq:der}--\eqref{eq:h} in $I=(t_0-T,t_0)$ for some $t_0\in\R$
  and $T>0$. Then there exist $m_0,k_0\in\N$,
  $k_0\geq 1$, such that
\begin{equation}\label{eq:671}
\lim_{t\to t_0^-}\widetilde{\mathcal
  N}(t)=\gamma_{m_0,k_0}=m_0-\frac{\alpha_{k_0}}2,
\end{equation}
where 
\begin{equation}\label{eq:almgren_quotient}
\widetilde{\mathcal N}(t)=\frac{
(t_0-t) \int_{\R^N}\big(|\nabla_\A  u(x,t)|^2-
  \frac{a({x}/{|x|})}{|x|^2}| u(x,t)|^2- h(x,t)|u(x,t)|^2\big)\, e^{-\frac{|x|^2}{4(t_0-t)}}\,dx}
{\int_{\R^N}|u(x,t)|^2\, e^{-\frac{|x|^2}{4(t_0-t)}}\,dx}.
\end{equation}
Furthermore, denoting as
$J_0$ the finite set of indices
$J_0=\{(m,k)\in\N\times(\N\setminus\{0\}):m-\frac{\alpha_{k}}2=\gamma_{m_0,k_0}\}$,
for all $\tau \in(0,1)$ there holds
\begin{equation}\label{eq:80}
\lim_{\lambda\to0^+}\int_\tau^1
\bigg\|\lambda^{-2\gamma_{m_0,k_0}}u(\lambda x,t_0-\lambda^2t)
-t^{\gamma_{m_0,k_0}}\sum_{(m,k)\in J_0}\beta_{m,k}\widetilde V_{m,k}(x/\sqrt t)
\bigg\|_{{\mathcal H}_t}^2dt=0
\end{equation}
and
\begin{equation}\label{eq:81}
\lim_{\lambda\to0^+}\sup_{t\in[\tau,1]}
\bigg\|\lambda^{-2\gamma_{m_0,k_0}}u(\lambda x,t_0-\lambda^2t)
-t^{\gamma_{m_0,k_0}}\sum_{(m,k)\in J_0}\beta_{m,k}\widetilde V_{m,k}(x/\sqrt t)
\bigg\|_{{\mathcal L}_t}=0,
\end{equation}
where
$\widetilde V_{m,k}=
V_{m,k}/\|V_{m,k}\|_{\mathcal L}$,
$V_{m,k}$ are as in (\ref{eq:66}),
\begin{multline}\label{eq:beta1}
\beta_{m,k}=\Lambda^{-2\gamma_{m_0,k_0}}
\int_{\R^N}u(\Lambda x,t_0-\Lambda^2)\overline{\widetilde V_{m,k}(x)}G(x,1)\,dx\\
+2\int_0^{\Lambda}s^{1-2\gamma_{m_0,k_0}}
\bigg(\int_{\R^N}h(s x,t_0- s^2) u(s x,t_0-s^2)\overline{\widetilde
  V_{m,k}(x)}
G(x,1)\,dx\bigg)
ds
\end{multline}
for all $\Lambda\in(0,\Lambda_0)$ and for some
$\Lambda_0\in(0,\sqrt{T})$, and $\beta_{m,k}\neq0$ for some $(m,k)\in J_0$.
\end{Theorem}

The effect of the magnetic singular potential on the
local behavior of solutions can be recognized in the values of the
limit frequencies $\gamma_{m,k}$ which are directly related to the
angular eigenvalues $\mu_k(\A,a)$ through formulas \eqref{eq:alfabeta}
and \eqref{eq:65}. We observe that the magnetic eigenvalues
$\mu_k(\A,a)$ are indeed different from the magnetic-free  eigenvalues
$\mu_k(0,a)$, at least in the case of a non irrotational magnetic
vector potential; e.g. in \cite[Lemma A.2]{FFT} it was proved that $\mu_1(\A,a)>
\mu_1(0,a)$ if $\mathop{\rm curl}(\A/|x|)\not\equiv 0$. We also recall
that, in the relevant example of a Aharonov-Bohm vector potential
\eqref{eq:72} in dimension $N=2$ (with $a\equiv0$), the magnetic eigenvalues are
explicitly known and the limit frequencies turns out to be
\[
m+\frac{|\alpha-k|}2,\quad m,k\in\N,
\]
showing how the asymptotic  behaviour of solutions strongly relies on
the circulation $\alpha$.

The proof of Theorem \ref{asym1} is based on a parabolic Almgren-Poon type 
monotonicity formula in the spirit of \cite{poon}  combined with a
blow-up analysis, see also \cite{FP}. In particular, the function
\eqref{eq:almgren_quotient}  
represents the Poon's electromagnetic parabolic counterpart  of the frequency quotient
introduced by Almgren in \cite{almgren}   
and used by Garofalo and Lin \cite{GL} to prove unique continuation
properties for elliptic equations with variable coefficients; indeed,
both in the elliptic and in the parabolic case, monotonicity
of the frequency function implies doubling properties of the solution
and then the validity of unique continuation principles. As done in
the proof of Theorem \ref{asym1} and as already observed in \cite{FP}
in the magnetic free case, the combination of
the monotonicity argument with a blow-up analysis allows proving not
only unique continuation but also the precise asymptotic description of
scaled solutions near the singularity given in
\eqref{eq:80}--\eqref{eq:81}. 
We notice that the magnetic free results of
\cite{FP} and their magnetic counterpart of the present paper
generalize the classification of local asymptotics of solutions to
parabolic equations with bounded coefficients obtained in \cite{CH} 
to the case of singular
homogenous potentials (Hardy potential and homogenous electromagnetic
potentials); we recall that the approach in \cite{CH}  is based on recasting
equations in terms of parabolic
self-similar variables. 
We also mention \cite{AV,E,EF,EKPV,EV,FE} for
unique continuation results for parabolic equations with
time-dependent potentials via Carleman inequalities and monotonicity
methods.

As a consequence of Theorem \ref{asym1} we obtain the following {\it strong unique
continuation property} at the singularity.

\begin{Corollary}\label{cor:uniq_cont}
  Let $u$ be a weak solution to (\ref{prob}) in $(t_0-T,t_0)$ under
  the assumptions of Theorem \ref{asym1}. If
\begin{equation}\label{eq:uniq_cont}
u(x,t)=O\Big(\Big(|x|^2+(t_0-t)\Big)^k\Big)\quad\text{as }x\to 0\text{ and
}t\to t_0^-,
\quad\text{for all }k\in\N,
\end{equation}
then $u\equiv 0$ in $\R^N\times(t_0-T,t_0)$.
\end{Corollary}

The monotonicity argument developed to prove Theorem
\ref{asym1} yields as a byproduct
the following {\it unique
continuation property} with respect to time.
It can be interpreted as a \emph{backward uniqueness result} for \eqref{prob} in the
spirit of \cite{lax}, see e.g. \cite{kukavica,lees-protter}.

\begin{Proposition}\label{p:uniq_cont}
Let
$N\geq2$,  ${\mathbf{A}}\in C^1({\mathbb{S}}
^{N-1},{\mathbb{R}}^N)$ and $a\in L^{\infty}\big({\mathbb S}^{N-1},\R\big)$ satisfy \eqref{transversality}
and
  \eqref{eq:hardycondition}. Let $u$ be a weak solution to
  \eqref{prob} in $\R^N\times I$ with  $h$ satisfying
  \eqref{eq:der}--\eqref{eq:h} for some bounded interval $I$.
If there exists $t_0\in I$
  such that
$$
u(x,t_0)=0\quad\text{for a.e. }x\in\R^N,
$$
then $u\equiv 0$ in $\R^N\times I$.
\end{Proposition}

Another main goal of this manuscript  is to give an integral  representation
formula for magnetic caloric functions, i.e. for solutions to
\eqref{prob}. 
The free heat forward equation, i.e. \eqref{prob} with
${\mathbf{A}}\equiv 0$, $a\equiv0$ and $h\equiv 0$, can be  considered as the canonical
example of  diffusion equation. A well-known solution to the
 the Cauchy problem
\begin{equation}  \label{eq:calorfree}
\begin{cases}
u_t=\Delta u \\
u(x,0)=f(x)
\end{cases}
\end{equation}
with datum $f\in C^0(\R^N)\cap L^\infty(\R^N)$ 
 is given by the integral formula:
\begin{equation*}
u(x,t)=e^{-t\Delta}f(x):=\frac{1}{(4\pi t)^{\frac N2}} 
\int_{{\mathbb{R}}^N} e^{-\frac{|x-y|^2}{4t}} f(y)\,dy, \quad 
x\in\R^N,\quad  t>0.
\end{equation*}
We also refer to \cite{Djrbashian} for integral representation
formulas for the heat equation in half--spaces.

With the aim of extending integral representation formulas
 to the electromagnetic singular case, the following theorem provides an explicit
representation formula for weak solutions to \eqref{prob}, with $h\equiv
0$ and initial datum $u_0\in \tilde{\mathcal L}$, where
$\tilde{{\mathcal L}}$ is defined as
the completion of $C^{\infty}_{\rm c}(\R^N,\C)$ with respect to
\begin{equation}\label{eq:49}
\|\varphi\|_{\tilde{{\mathcal L}}}=\bigg(\int_{\R^N}|\varphi(x)|^2 e^{\frac{|x|^2}{4}}\,dx\bigg)^{\!\!1/2}.
\end{equation}

\begin{Theorem}\label{t:representation}
Let $u$
be a weak solution (in the sense explained at the beginning of Section \ref{sec:repr-form-solut})
to (\ref{prob})
with $h\equiv 0$ and $u(\cdot,0)=u_0\in
\tilde{\mathcal L}$. Then $u$ admits the following representation:
for all $t> 0$,
\[
  u(x,t)=t^{-\frac{N}{2}}
  \int_{\R^N}u_{0}(y)  K\bigg(\frac{y}{\sqrt{t}},\frac{x}{\sqrt{t}}\bigg) \,dy, 
  \]
where the integral at the right hand side is understood in the sense
of  improper multiple integrals and 
\begin{equation}\label{eq:3}
K(x,y)=
\frac{1}{2(|x||y|)^{\frac{N-2}{2}}}
e^ {-\frac{|x|^2+|y|^2}{4}}
\sum\limits_{k=1}^{\infty
}e^{i \frac\pi2\beta _{k}}\psi _{k}\big(\tfrac{y}{|y|}\big) 
\overline{\psi _{k}\big(\tfrac{x}{|x|}\big)}J_{\beta _{k}}\big(  \tfrac{-i|x||y|
}{2}\big)
\end{equation}
being $\beta_k$ as in \eqref{eq:alfabeta}, $\psi_k$ as in
\eqref{angular}, and being $J_{\beta _{k}}$   the Bessel function of
the first kind of order 
$\beta _{k}$.
  \end{Theorem}
In the proof of Theorem \ref{t:representation},
the critical homogeneities and the transversality condition
\eqref{transversality} play a fundamental role. Indeed  Theorem
\ref{t:representation}  is proved by recasting
equation \eqref{prob} with $h=0$ in terms of parabolic
self-similar variables (see transformation \eqref{varphi}) thus
obtaining an equivalent parabolic equation with an Ornstein-Uhlenbeck type 
operator with singular homogeneous electromagnetic
potentials (see \eqref{varphieq}). Then a representation formula
 is obtained by expanding the transformed solution in Fourier
series with respect to an orthonormal basis given by eigenfunctions of
the
 Ornstein-Uhlenbeck type 
operator (see Remark \ref{rem:tilde} for the description of such a basis).

The present paper is organized as follows. In section
\ref{sec:weak-form-probl} we give a weak formulation of problem
\eqref{prob}. In section \ref{sec:parabolic-hardy-type} we present 
some magnetic parabolic Hardy type
  inequalities and weighted Sobolev embeddings. Section
  \ref{sec:spectr-analys-self} is devoted to the description 
of the spectrum of the operator $L_{\A,a}$ defined in \eqref{eq:13}
and to the proof of Proposition \ref{p:explicit_spectrum}.
The parabolic monotonicity argument developed in section
\ref{sec:almgren} together with the blow-up analysis of section
\ref{sec:blow-up-analysis} allow proving Theorem \ref{asym1} at the
end of section
\ref{sec:blow-up-analysis}. Finally, section \ref{sec:repr-form-solut}
contains the proof of the representation formula stated in Theorem~\ref{t:representation}.

\medskip
\noindent
{\bf Notation. } We list below some notation used throughout the
paper.\par
\begin{itemize}
\item[-] ${\rm const}$ denotes some positive constant which may vary
  from formula to formula.
\item[-] $dS$ denotes the volume element on the unit
  $(N-1)$-dimensional sphere ${\mathbb S}^{N-1}$.
\item[-] $\omega_{N-1}$ denotes the volume of  ${\mathbb
  S}^{N-1}$, i.e. $\omega_{N-1}=\int_{{\mathbb S}^{N-1}}dS(\theta)$.
\item[-] For every complex number $z$ we denote as $\Re(z)$ its real part.
\end{itemize}

\section{The weak formulation of the problem}\label{sec:weak-form-probl}
The functional space $\mathcal H_t$ defined in \eqref{eq:star} is
related to the weighted magnetic Sobolev space ${\mathcal H}^{\A}_t$
 defined,  for every
$t>0$, as
the completion of $C^{\infty}_{\rm c}(\R^N\setminus \{0\},\C)$ with
respect to the norm 
\[
\|u\|_{{\mathcal H}_t^\A}=\bigg(\int_{\R^N}\big(t|\nabla\!_\A
u(x)|^2+|u(x)|^2\big) G(x,t)\,dx\bigg)^{\!\!1/2},
\]
where $\nabla_{{\mathbf{A}}}u= \nabla u+i\frac {{\mathbf{A}}(x/|x|)}{
|x|}u$.
For $\A\equiv 0$ we recover the Hilbert Space 
\[
{H}_t:={\mathcal H}_t^0
\]
 described in \cite{FP}. 
From \cite[Proposition 3.1]{poon} (see Lemma \ref{Hardytemp} in
section \ref{sec:parabolic-hardy-type})
it follows that 
\[
\text{if $N\geq 3$ then  ${\mathcal{H}_t
}={{H}_t }=\mathcal H^0_t$ and the norms
$\|\cdot\|_{\mathcal{H}_t}$, $\|\cdot\|_{{H}_t}$ are
equivalent.}
\]
On the other hand, if $N=2$ we have that ${\mathcal{H}_t
}\subset{{H}_t }$ and ${\mathcal{H}_t
}\not={{H}_t }$. 

We also notice that $\mathcal{H}_t$ is continuously embedded into
  ${\mathcal H}_t^\A$; a detailed comparison between spaces
  $\mathcal{H}_t$ and ${\mathcal H}_t^\A$ will be performed in
    section \ref{sec:parabolic-hardy-type}.

We denote as ${\mathcal H}_t^\star$ the dual space of
${\mathcal H}_t$ and by ${}_{{\mathcal H}_t^\star}\langle
\cdot,\cdot\rangle_{{\mathcal H}_t}$ the corresponding duality product.

In order to prove Theorem \ref{asym1}, 
up to a translation it is not restrictive to assume that 
\[
t_0=0.
\]
Furthermore, to simplify notations and work with positive times $t$, 
it is convenient to perform the change of variable $(x,t)\mapsto
(x,-t)$. Indeed, if $u(x,t)$ is solution to (\ref{prob}) in $\R^N\times
(-T,0)$, then $\tilde u(x,t)= u(x,-t)$ solves
\begin{equation}\label{probtilde}
-\tilde{u}_t(x,t)+{\mathcal{L}}_{{\mathbf{A}},a} \tilde{u} (x,t) = h(x,-t) \tilde{u}(x,t) \inn  \R^N\times (0,T).
\end{equation}
\begin{Definition}\label{def:solution}
  We say that $\tilde{u}\in L^1_{\rm loc
  }(\R^N\times(0,T))$ is a weak solution to
  (\ref{probtilde}) in $\R^N\times(0,T)$ if
\begin{align}
  \label{eq:defsol1}&\int_\tau^T\|\tilde{u}(\cdot,t)\|^2_{{\mathcal
      H}_t}\,dt<+\infty,\quad\int_\tau^T\Big\|\tilde{u}_t+\frac{\nabla_\A \tilde{u}\cdot
                      x}{2t}\Big\|^2_{{\mathcal H}_t^\star}dt<+\infty \text{ for all
  }\tau\in (0,T),\\
\label{eq:defsol2}&{\phantom{\bigg\langle}}_{{\mathcal
    H}_t^\star}\bigg\langle
\tilde{u}_t+\frac{\nabla_\A \tilde{u}\cdot x}{2t},\phi
\bigg\rangle_{{\mathcal H}_t}\\
&\notag\qquad=
\int_{\R^N}\bigg(\nabla_\A \tilde{u}(x,t)\cdot \overline{\nabla_\A \phi(x)}-
\dfrac{a(x/|x|)}{|x|^2}\,\tilde{u}(x,t) \overline{\phi(x)}-h(x,-t)\tilde{u}(x,t) \overline{\phi(x)}\bigg)G(x,t)\,dx
\end{align}
for a.e. $t\in (0,T)$ and for each $\phi\in {\mathcal
    H}_t$.
\end{Definition}

\begin{remark}\label{rem:uv}
In view of \eqref{transversality} we have that $\nabla_\A \tilde
u\cdot x=\nabla\tilde u\cdot x$. Therefore, if $\tilde{u}\in L^1_{\rm loc }(\R^N\times(0,T))$ satisfies (\ref{eq:defsol1}), then
the function
\begin{equation}\label{eq:5}
v(x,t):=\tilde{u}(\sqrt{t}x,t)
\end{equation}
satisfies
\begin{equation}\label{eq:4}
v\in L^2(\tau,T;\mathcal H)\quad\text{and}\quad
v_t\in L^2(\tau,T;\mathcal H^\star) \quad\text{for all
  }\tau\in (0,T),
\end{equation}
where $\mathcal H:={\mathcal H}_1$ is defined in \eqref{eq:Hstar}.
From (\ref{eq:4}) it follows that $v\in C^0([\tau,T],{\mathcal L})$
(see e.g. \cite{SH}),
being $\mathcal L:={\mathcal L}_1$, i.e. $\mathcal L$ is 
the completion of $C^{\infty}_{\rm c}(\R^N,\C)$ with respect to the norm
$\|v\|_{{\mathcal L}}=\big(\int_{\R^N}|v(x)|^2e^{-|x|^2/4}\,dx\big)^{1/2}$.
Moreover the function
$$
t\in[\tau,T]\mapsto \|v(t)\|^2_{{\mathcal L}}=\int_{\R^N}|\tilde u(x,t)|^2G(x,t)\,dx
$$
is absolutely continuous  and
\begin{align*}
\frac12\frac1{dt} \int_{\R^N}|\tilde u(x,t)|^2G(x,t)&=
\frac12\frac1{dt}\|v(t)\|_{{\mathcal L}}^2=
\Re\Big[{}_{{\mathcal H}^\star}\langle
v_t(\cdot,t),v(\cdot,t)
\rangle_{{\mathcal H}}\Big]\\
&=\Re\bigg[\!\!\!{\phantom{\bigg\langle}}_{{\mathcal
    H}_t^\star\!\!}\bigg\langle
\tilde u_t+\frac{\nabla_\A \tilde u\cdot x}{2t},\tilde u(\cdot,t)
\bigg\rangle_{{\mathcal H}_t}\bigg]
\end{align*}
for a.e. $t\in(0,T)$.
\end{remark}
\begin{remark}\label{rem:v2}
  If $u$ is a weak solution to (\ref{prob}) in the sense of definition
  \ref{def:solution}, then the function $v(x,t):=\tilde{u}(\sqrt{t}x,t)$
  defined in \eqref{eq:5} is  a weak solution to
\begin{equation*}
  v_t+\frac1t\bigg(-{\mathcal{L}}_{{\mathbf{A}},a}  v-\frac x2\cdot \nabla v+t h(\sqrt tx, -t)v\bigg)=0,
\end{equation*}
in the sense that, for every $\phi\in{\mathcal H}$,
\begin{multline}\label{eq:24}
{\phantom{\big\langle}}_{{\mathcal H}^\star}\!\big\langle v_t,\phi \big\rangle_{{\mathcal H}}\\=
\frac1t\int_{\R^N}\!\!\bigg(\!\nabla_\A v(x,t)\!\cdot \!\overline{\nabla_\A \phi(x)}-
\dfrac{a\big(\frac{x}{|x|}\big)}{|x|^2}\,v(x,t) \overline{\phi(x)} -t\,h(\sqrt
tx,-t)v(x,t) \overline{\phi(x)}\!\bigg)G(x,1)\,dx.
\end{multline}
\end{remark}

\section{Magnetic parabolic Hardy type
  inequalities and weighted Sobolev embeddings}\label{sec:parabolic-hardy-type}
The following Hardy type inequality for parabolic
operators was proved in \cite[Proposition 3.1]{poon}.

\begin{Lemma} \label{Hardytemp}
  For every $t>0$, $N\geq 3$ and $u\in {H}_t={\mathcal H}^0_t$ there holds
$$
\int_{\R^N}\dfrac{|u(x)|^2}{|x|^2}\,G(x,t)\,dx\leq
\dfrac{1}{(N-2)t}\dint_{\R^N}|u(x)|^2G(x,t)\,dx+
\dfrac{4}{(N-2)^2}\dint_{\R^N}|\nabla
u(x)|^2\,G(x,t)\,dx.
$$
\end{Lemma}

In order to compare the space ${\mathcal{H}}_t$ with the magnetic
  space ${\mathcal{H}^{\A}_t}$, we recall the well-known
    \emph{diamagnetic inequality}:  for all
$u\in \mathcal H^\A_t$,
\begin{equation}\label{eq:9}
|\nabla |u|(x)|\leq
\big|\nabla_\A u(x)\big|, \quad \text{for a.e. }  x\in \R^N.
\end{equation}
Lemma \ref{Hardytemp} and the diamagnetic inequality \eqref{eq:9}
easily imply that 
\[
\text{if $N\geq 3$ and ${\mathbf{A}}\in C^1({\mathbb{S}}
  ^{N-1},{\mathbb{R}}^N)$ then } \mathcal H^\A_t= H_t=\mathcal H_t\text{
  for all }t>0.
\]
The following lemma extends the Hardy type inequality of Lemma
\ref{Hardytemp} to the electromagnetic case; we notice that the
presence of an  electromagnetic potential satisfying the positivity
condition \eqref{eq:hardycondition} allows recovering a Hardy
inequality even in dimension $2$. 
\begin{Lemma}\label{Hardy_aniso}
Let  $N\geq 2$, $a\in L^{\infty}\big({\mathbb S}^{N-1},\R\big)$ and
 let ${\mathbf{A}}\in C^1({\mathbb{S}} ^{N-1},{\mathbb{R}}^N)$ satisfy
  the  transversality condition \eqref{transversality}. For every $t>0$ and
  $u\in {\mathcal H}_t$, there holds
\begin{multline*}
\int_{\R^N}\bigg(|\nabla_{\A}
u(x)|^2-\frac{a(x/|x|)}{|x|^2}\,|u(x)|^2\bigg)
\,G(x,t)\,dx+\frac{N-2}{4t}\int_{\R^N}|u(x)|^2G(x,t)\,dx\\
\geq \bigg(\mu_1(\A,a)+\frac{(N-2)^2}4\bigg)
\int_{\R^N}\dfrac{|u(x)|^2}{|x|^2}\,G(x,t)\,dx.
\end{multline*}
\end{Lemma}
\begin{pf}
Let $u\in C^\infty_{\rm
  c}(\R^N\setminus\{0\},\C)$.  The magnetic gradient of $u$ can be
written in polar coordinates as
\begin{displaymath}
  \nabla_\A u(x)= \left(\nabla +i\dfrac{\A(\theta)}{|x|}\right)u=\big(\partial_ru(r\theta)\big)\theta+
  \frac 1r\nabla_{{\mathbb S}^{N-1}}u(r\theta)+i\dfrac{\A(\theta)}{r} u(r\theta),
\quad r=|x|,\  \theta=\frac{x}{|x|},
\end{displaymath}
hence, in view of \eqref{transversality},
\begin{displaymath}
     |\nabla_\A u(x)|^2
=\big|\partial_ru(r\theta)\big|^2+
\frac1{r^2}\big|(\nabla_{{\mathbb S}^{N-1}}+i\A)u(r\theta)
\big|^2.
  \end{displaymath}
Hence
\begin{multline}\label{eq:1}
  \int_{\R^N}\bigg(|\nabla_\A
  u(x)|^2-\frac{a(x/|x|)}{|x|^2}\,|u(x)|^2\bigg) \,G(x,t)\,dx\\=
  t^{-\frac N2}\int_{{\mathbb S}^{N-1}}\bigg(
  \int_0^{+\infty}r^{N-1}e^{-\frac{r^2}{4t}}
  |\partial_ru(r\theta)|^2\,dr\bigg)\,dS(\theta)\\
  +t^{-\frac N2}
  \int_0^{+\infty}\frac{r^{N-1}e^{-\frac{r^2}{4t}}}{r^2}\bigg(\int_{{\mathbb
      S}^{N-1}}\left[|(\nabla_{{\mathbb
        S}^{N-1}}+i\A)u(r\theta)|^2-a(\theta)
    |u(r\theta)|^2\right]\,dS(\theta)\bigg)\,dr.
\end{multline}
For all $\theta\in{\mathbb S}^{N-1}$, let $\varphi_{\theta}\in
C^\infty_{\rm c}((0,+\infty),\C)$ be defined as
$\varphi_{\theta}(r)=u(r\theta)$, and $\widetilde\varphi_{\theta}\in
C^\infty_{\rm c}(\R^N\setminus\{0\},\C)$ be the radially symmetric
function given by
$\widetilde\varphi_{\theta}(x)=\varphi_{\theta}(|x|)$.
If $N\geq3$, from Lemma \ref{Hardytemp}, it follows that
\begin{align}\label{eq:2b}
  t^{-\frac N2} \int_{{\mathbb S}^{N-1}}&\bigg(
  \int_0^{+\infty}r^{N-1}e^{-\frac{r^2}{4t}}
  |\partial_ru(r\theta)|^2\,dr\bigg)\,dS(\theta)\\
  \notag& = t^{-\frac
    N2}\int_{{\mathbb S}^{N-1}}\bigg( \int_0^{+\infty}r^{N-1}
  e^{-\frac{r^2}{4t}}|\varphi_{\theta}'(r)|^2\,dr\bigg)\,dS(\theta)\\
  \notag& =\frac1{\omega_{N-1}} \int_{{\mathbb S}^{N-1}}\bigg(
  \int_{\R^N}|\nabla\widetilde
  \varphi_{\theta}(x)|^2G(x,t)\,dx\bigg)\,dS(\theta) \\
  &\notag\geq \frac1{\omega_{N-1}} \frac{(N-2)^2}{4} \int_{{\mathbb
      S}^{N-1}}\bigg(\int_{\R^N}\frac{|\widetilde
    \varphi_{\theta}(x)|^2}{|x|^2}G(x,t)\,dx\bigg)\,dS(\theta)\\
  &\notag \quad-\frac1{\omega_{N-1}}\frac{N-2}{4t}\int_{{\mathbb
      S}^{N-1}}\bigg(\int_{\R^N}|\widetilde
  \varphi_{\theta}(x)|^2G(x,t)\,dx\bigg)\,dS(\theta)
  \\
  \notag&=t^{-\frac N2}\frac{(N-2)^2}4 \int_{{\mathbb S}^{N-1}}\bigg(
  \int_0^{+\infty}\frac{r^{N-1}e^{-\frac{r^2}{4t}}}{r^2}|u(r\theta)|^2\,dr\bigg)
  \,dS(\theta)\\
  \notag&-t^{-\frac N2}\frac{N-2}{4t}\int_{{\mathbb S}^{N-1}}\bigg(
  \int_0^{+\infty}r^{N-1}e^{-\frac{r^2}{4t}}|u(r\theta)|^2\,dr\bigg)\,dS(\theta)\\
  \notag&
  =
\frac{(N-2)^2}4
\int_{\R^N}\dfrac{|u(x)|^2}{|x|^2}\,G(x,t)\,dx-
\frac{N-2}{4t}\int_{\R^N}|u(x)|^2G(x,t)\,dx,
\end{align}
where $\omega_{N-1}$ denotes the volume of the unit sphere ${\mathbb
  S}^{N-1}$, i.e. $\omega_{N-1}=\int_{{\mathbb
    S}^{N-1}}dS(\theta)$. If $N=2$ we have
trivially that 
\begin{equation}\label{eq:10}
  t^{-\frac N2} \int_{{\mathbb S}^{N-1}}\bigg(
  \int_0^{+\infty}r^{N-1}e^{-\frac{r^2}{4t}}
  |\partial_ru(r\theta)|^2\,dr\bigg)\,dS(\theta)\geq0.
\end{equation}
On the other hand, from the definition of $\mu_1(\A,a)$ it follows
that
\begin{equation} \label{eq:3b}
  \int_{{\mathbb
      S}^{N-1}}\!\!\left[|(\nabla_{{\mathbb S}^{N-1}}+i\A(\theta))u(r\theta)|^2\!-
a(\theta)|u(r\theta)|^2\right]dS(\theta)
  \geq \mu_1(\A,a)\int_{{\mathbb S}^{N-1}}\!|u(r\theta)|^2dS(\theta).
\end{equation}
From (\ref{eq:1}), (\ref{eq:2b}), \eqref{eq:10}, and (\ref{eq:3b}), we
deduce the stated inequality 
for all $u\in C^\infty_{\rm
    c}(\R^N\setminus\{0\},\C)$. The conclusion follows by  density of $C^\infty_{\rm c}(\R^N\setminus\{0\},\C)$
in ${\mathcal H}_t$.~\end{pf}

Lemma \ref{Hardy_aniso} allows extending the Hardy type inequality of
Lemma \ref{Hardytemp} to the case $N=2$ in the presence of a vector
potential satisfying a suitable non-degeneracy condition. Indeed 
Lemma \ref{Hardy_aniso} implies that, if  $N=2$, $t>0$ and
  $u\in {\mathcal H}_t$, then
\begin{equation}\label{eq:12}
\int_{\R^N}|\nabla_{\A}
u(x)|^2\,G(x,t)\,dx\geq \mu_1(\A,0)
\int_{\R^N}\dfrac{|u(x)|^2}{|x|^2}\,G(x,t)\,dx,
\end{equation}
thus giving a  Hardy type inequality if $\mu_1(\A,0)>0$. As observe in
\cite[Section 2]{FFT}, the condition  $\mu_1(\A,0)>0$ holds if and
only if 
\begin{equation}  \label{eq:circuit}
\Phi_{\mathbf{A}}:=\frac1{2\pi}\int_0^{2\pi}\alpha(t)\,dt \not\in{\mathbb{Z}}%
,\quad \text{where }\alpha(t):={\mathbf{A}}(\cos t,\sin t)\cdot(-\sin t,\cos
t).
\end{equation}
Condition \eqref{eq:circuit} is related the following Hardy
inequality
 proved in \cite{lw}:
\begin{equation}  \label{eq:hardyN2}
\Big(\min_{k\in{\mathbb{Z}}}|k-\Phi_{\mathbf{A}}|\Big)^2\int_{{\mathbb{R}}^2}%
\frac{|u(x)|^2}{|x|^2}\,dx \leq \int_{{\mathbb{R}}^2}\bigg|\nabla u(x)+i\,
\frac{{\mathbf{A}}\big({x}/{|x|}\big)} {|x|}\,u(x)\bigg|^2\,dx
\end{equation}
which holds for all functions $u\in C^\infty_{\rm
  c}(\R^N\setminus\{0\},\C)$. 
Moreover
$\big(\min_{k\in{\mathbb{Z}}}|k-\Phi_{\mathbf{A}}|\big)^2=\mu_1(\A,0)$
is the best
constant in \eqref{eq:hardyN2}.

From \eqref{eq:12} it follows that
\[
\text{if $N=2$ and
\eqref{transversality} and \eqref{eq:circuit} hold, then that
$\mathcal{H}_{t}={\mathcal{H}}_{t}^\A$},
\] 
being the norms $\|\cdot\|_{{\mathcal{H}_{t}}}$, $\|\cdot\|_{{\mathcal{H}_t^\A}}$ equivalent.
On the other hand, if $N=2$ and $\Phi_\A\in\Z$
(i.e. $\mu_1(\A,0)=0$), then $\A$ is gauge equivalent to $0$ and 
$\mathcal{H}_{t}\subset{\mathcal{H}}_{t}^\A$,
$\mathcal{H}_{t}\neq{\mathcal{H}}_{t}^\A$; this case is actually not
very interesting since it can be reduced to the magnetic free problem by
a gauge transformation.

The following corollary provides a positivity condition for the
quadratic form associated to the electromagnetic potential under
condition  \eqref{eq:hardycondition}.
\begin{Corollary}\label{c:pos_def}
Let  $N\geq 2$, $a\in L^{\infty}\big({\mathbb S}^{N-1},\R\big)$ and let
  ${\mathbf{A}}\in C^1({\mathbb{S}} ^{N-1},{\mathbb{R}}^N)$ satisfy
   \eqref{transversality} and  \eqref{eq:hardycondition}.
Then, for every $t>0$,
\begin{align*}
  &\inf_{u\in {\mathcal
      H}_t\setminus\{0\}}\frac{\int_{\R^N}\big(|\nabla_\A
    u(x)|^2-\frac{a(x/|x|)}{|x|^2}\,|u(x)|^2\big)
    \,G(x,t)\,dx+\frac{N-2}{4t}\int_{\R^N}|u(x)|^2G(x,t)\,dx}
  {\int_{\R^N}|\nabla u(x)|^2
    \,G(x,t)\,dx+\int_{\R^N}\frac{|u(x)|^2}{|x|^2}
    \,G(x,t)\,dx+\frac{N-2}{4t}\int_{\R^N}|u(x)|^2G(x,t)\,dx}\\[5pt]
  &\quad=\inf_{v\in {\mathcal
      H}\setminus\{0\}}\frac{\int_{\R^N}\big(|\nabla_\A
    v(x)|^2-\frac{a(x/|x|)}{|x|^2}\,|v(x)|^2\big)
    \,G(x,1)\,dx+\frac{N-2}{4}\int_{\R^N}|v(x)|^2G(x,1)\,dx}
  {\int_{\R^N}|\nabla v(x)|^2
    \,G(x,1)\,dx+\int_{\R^N}\frac{|v(x)|^2}{|x|^2}
    \,G(x,1)\,dx+\frac{N-2}{4}\int_{\R^N}|v(x)|^2G(x,1)\,dx}>0.
\end{align*}
\end{Corollary}
\begin{pf}
  The change of  variables $u(x)=v(x/\sqrt t)$ immediately gives the
  equality of the two infimum levels.
To prove
  that they are strictly positive, we argue by contradiction. Let us 
  assume that for every $\e>0$ there exists $v_\e\in {\mathcal
    H}\setminus\{0\}$ such that
\begin{multline*}
\int_{\R^N}\bigg(|\nabla_\A
v_\e(x)|^2-\frac{a(x/|x|)}{|x|^2}\,|v_\e(x)|^2\bigg)
\,G(x,1)\,dx+\frac{N-2}{4}\int_{\R^N}|v_\e(x)|^2G(x,1)\,dx\\
<\e \bigg(
\int_{\R^N}|\nabla v_\e(x)|^2
    \,G(x,1)\,dx+\int_{\R^N}\frac{|v_\e(x)|^2}{|x|^2}
    \,G(x,1)\,dx+\frac{N-2}{4}\int_{\R^N}|v_\e(x)|^2G(x,1)\,dx\bigg).
\end{multline*}
Hence 
\begin{multline*}
\int_{\R^N}\bigg(|\nabla_\A
v_\e(x)|^2-\frac{a(x/|x|)}{|x|^2}\,|v_\e(x)|^2\bigg)
\,G(x,1)\,dx+\frac{N-2}{4}\int_{\R^N}|v_\e(x)|^2G(x,1)\,dx\\
<\e 
\int_{\R^N}\left(|\nabla_\A v_\e(x)|^2+(1+\|\A\|_{L^\infty({\mathbb S}^{N-1})})\frac{|v_\e(x)|^2}{|x|^2}
+\frac{N-2}{4}|v_\e(x)|^2\right)G(x,1)\,dx\\
\leq \e (1+\|\A\|_{L^\infty({\mathbb S}^{N-1})})
\int_{\R^N}\left(|\nabla_\A v_\e(x)|^2+\frac{|v_\e(x)|^2}{|x|^2}
+\frac{N-2}{4}|v_\e(x)|^2\right)G(x,1)\,dx.
\end{multline*}
From the above inequality and Lemma \ref{Hardy_aniso} it follows that
\begin{align*}
  \bigg(\mu_1\bigg(\A,&\frac{a}{1-(1+\|\A\|_{L^\infty({\mathbb S}^{N-1})})\e}\bigg)+\frac{(N-2)^2}4\bigg)
  \int_{\R^N}\dfrac{|v_\e(x)|^2}{|x|^2}\,G(x,1)\,dx\\
&  \leq \int_{\R^N}\bigg(|\nabla_\A
  v_\e(x)|^2-\frac{a(x/|x|)}{1-(1+\|\A\|_{L^\infty({\mathbb S}^{N-1})})\e}\,\frac{|v_\e(x)|^2}{|x|^2}\bigg)
  \,G(x,1)\,dx\\
&\qquad\qquad+\frac{N-2}{4}\int_{\R^N}|v_\e(x) |^2 G(x,1)\,dx\\
&<\frac{(1+\|\A\|_{L^\infty({\mathbb
      S}^{N-1})})\e}{1-(1+\|\A\|_{L^\infty({\mathbb S}^{N-1})})\e}
\int_{\R^N}\frac{|v_\e(x)|^2}{|x|^2}\,G(x,1)\,dx,
\end{align*}
hence
$$
\mu_1\bigg(\A,\frac{a}{1-(1+\|\A\|_{L^\infty({\mathbb
      S}^{N-1})})\e}\bigg)+\frac{(N-2)^2}4
<\frac{(1+\|\A\|_{L^\infty({\mathbb
      S}^{N-1})})\e}{1-(1+\|\A\|_{L^\infty({\mathbb S}^{N-1})})\e}.
$$
From \eqref{firsteig} it follows easily that, for  fixed
${\mathbf{A}}\in C^1({\mathbb{S}} ^{N-1},{\mathbb{R}}^N)$, the map
$a\mapsto \mu_1(\A,a)$ 
is continuous with respect to the
$L^{\infty}\big({\mathbb S}^{N-1}\big)$-norm. Therefore, letting
$\e\to0$ in the
above inequality, we obtain $\mu_1(\A,a)+\frac{(N-2)^2}4\leq0$, thus
contradicting assumption \eqref{eq:hardycondition}.
\end{pf}

The negligibility assumption 
\eqref{eq:h} allows treating $h$ as a lower order potential,
recovering for small time the positivity of the quadratic form
associated to \eqref{probtilde}.
\begin{Corollary}\label{c:pos_per}
Let  $N\geq 2$, $a\in L^{\infty}\big({\mathbb S}^{N-1},\R\big)$ and
let 
  ${\mathbf{A}}\in C^1({\mathbb{S}} ^{N-1},{\mathbb{R}}^N)$ satisfy
   \eqref{transversality} and  \eqref{eq:hardycondition}.
Let $T>0$, $I=(-T,0)$ and $h\in L^\infty_{{\rm loc}}((\R^N\setminus
  \{0\})\times (-T,0))$ satisfy (\ref{eq:h}) in $I$. Then there exist $C_1,C_2>0$ and
  $\overline{T}>0$ such that for every $t\in(0,\overline{T})$,
$s\in(-T,0)$, and $u\in
  {\mathcal H}_t$ 
\begin{align*}
\int_{\R^N}\bigg(|\nabla_\A
u(x)|^2&-\frac{a(x/|x|)}{|x|^2}\,|u(x)|^2-|h(x,s)||u(x) |^2\bigg)
\,G(x,t)\,dx\\
&\geq C_1\int_{\R^N}\frac{|u(x) |^2}{|x|^2}\,G(x,t)\,dx
-\frac{C_2}t\int_{\R^N}|u(x) |^2 G(x,t)\,dx\\
\int_{\R^N}\bigg(|\nabla_\A
u(x)|^2&-\frac{a(x/|x|)}{|x|^2}\,|u(x) |^2-|h(x,s)||u(x) |^2\bigg)
\,G(x,t)\,dx
+\frac{N-2}{4t}\int_{\R^N}|u(x) |^2 G(x,t)\,dx\\
&\geq C_1\bigg(\int_{\R^N}|\nabla
u(x)|^2\,G(x,t)\,dx+\int_{\R^N}\frac{|u(x)|^2}{|x|^2}
\,G(x,t)\,dx+\frac1t\int_{\R^N}|u(x) |^2 G(x,t)\,dx\bigg).
\end{align*}
\end{Corollary}
\begin{pf}
From (\ref{eq:h}) it follows that, for every $u\in
  {\mathcal H}_t$,
\begin{align}\label{eq:11}
&  \left| \int_{\R^N}|h(x,s)||u(x)|^2G(x,t)\,dx\right| \\
&\notag\leq C_h\bigg(
  \int_{\R^N}|u(x)|^2G(x,t)\,dx+\int_{\R^N}|x|^{-2+\e}|u(x)|^2G(x,t)\,dx\bigg)\\
\notag  &\leq C_h\bigg(
  \int_{\R^N}|u(x)|^2G(x,t)\,dx+t^{\e/2}\int\limits_{|x|\leq\sqrt t
  }\frac{|u(x)|^2}{|x|^2}G(x,t)\,dx
\\
&\notag\qquad\qquad+t^{-1+\e/2}\int\limits_{|x|\geq\sqrt t
  }|u(x)|^2G(x,t)\,dx
\bigg)\\
\notag&=\frac{C_h}{t}(t+t^{\e/2})\int_{\R^N}|u(x)|^2G(x,t)\,dx+
C_ht^{\e/2}\int_{\R^N}\frac{|u(x)|^2}{|x|^2}G(x,t)\,dx.
\end{align}
The stated inequalities follow from (\ref{eq:11}),
Lemma \ref{Hardy_aniso}, Corollary \ref{c:pos_def},
and assumption  (\ref{eq:hardycondition}).
\end{pf}

The proof of the following
inequality  follows the spirit of \cite[Lemma 3]{EFV}.

\begin{Lemma}\label{l:ineqx2}
For every $u\in {\mathcal H}$, $|x|u\in{\mathcal L}$ and
$$
\frac{1}{16}
\int_{\R^N}|x|^2|u(x)|^2G(x,1)\,dx\leq
\int_{\R^{N}}|\nabla_\A u(x)|^2G(x,1)\,dx+\frac{N}{4}\int_{\R^N}|u(x)|^2G(x,1)\,dx.
$$
\end{Lemma}
\begin{pf}
  Let $u \in C^\infty_{\rm c}(\R^N\setminus\{0\},\C)$ and
  $w=e^{-\frac{|x|^2}{8}} u$. It follows that
$$
\nabla_{\A} w= e^{-\frac{|x|^2}{8}} \Bigg(\bigg(\nabla +i\frac{\A}{|x|}\bigg) u - \frac{ux}{4}\Bigg).
$$
Hence, by the transversality condition \eqref{transversality}, an
integration by parts yields that
\begin{align*}
\int_{\mathbb{R}^N} |\nabla_\A w|^2\, dx&=
 \int_{\mathbb{R}^N} e^{-\frac{|x|^2}{4}}  \left( \left|\big(\nabla
+i\tfrac{\A}{|x|}\big) u\right|^2 +\dfrac{|x|^2
|u|^2}{16} - \dfrac{1}{2}\Re \left(\left(\nabla +i\tfrac{\A}{|x|}\right)
                                          u \cdot x\,\bar{u}
                                          \right)
\right)\,dx\\
&=  \dint_{\mathbb{R}^N} e^{-\frac{|x|^2}{4}}  \left(  \left|\big(\nabla
+i\tfrac{\A}{|x|}\big) u\right|^2+\dfrac{|x|^2 |u|^2}{16} - \dfrac{1}{4} \nabla |u|^2\cdot x\right)\,dx\\
&= \dint_{\mathbb{R}^N} e^{-\frac{|x|^2}{4}}  |\nabla_\A u|^2\,dx -\dfrac{1}{16}\dint_{\mathbb{R}^N} e^{-\frac{|x|^2}{4}} {|x|^2 |u|^2}\,dx + \dfrac{N}{4} \dint_{\mathbb{R}^N}  e^{-\frac{|x|^2}{4}} |u|^2\, dx\geq 0.
\end{align*}
The conclusion follows by  density of $C^\infty_{\rm c}(\R^N\setminus\{0\},\C)$
in ${\mathcal H}$.
\end{pf}

Using the change of variables $u(x)=v(x/\sqrt t)$, it is easy to
verify that Lemma \ref{l:ineqx2}
implies the following inequality in ${\mathcal H}_t$.
\begin{Corollary}\label{cor:ineq}
For every $t>0$ and $u\in {\mathcal H}_t$, there holds
$$
\frac{1}{16t^2}
\int_{\R^N}|x|^2|u(x)|^2G(x,t)\,dx\leq
\int_{\R^{N}}|\nabla_\A u(x)|^2G(x,t)\,dx+\frac{N}{4t}\int_{\R^N}|u(x)|^2G(x,t)\,dx.
$$
\end{Corollary}

The following weighted Sobolev type inequalities hold. 

\begin{Lemma}\label{l:sob}
For all $v\in {\mathcal H}$  there holds
 $v\sqrt{G(\cdot,1)}\in L^{s}(\R^N)$ for all $s\in[2,\frac{2N}{N-2}]$
 if $N\geq3$ and for all $s\geq2$ if $N=2$; furthermore 
$$
\bigg(
\int_{\R^N}|v(x)|^sG^{\frac{s}{2}}(x,1)\,dx\bigg)^{\!\!\frac{2}{s}}\leq
C_s\|v\|^2_{{\mathcal H}}
$$
for some constant $C_s>0$ independent of $v\in {\mathcal H}$.
Moreover, for every $t>0$ and $u\in {\mathcal H}_t$,
$$
\Big(
\int_{\R^N}|u(x)|^{s}G^{\frac{s}{2}}(x,t)\,dx\Big)^{\!\!\frac{2}{s}}\leq C_s
t^{-\frac{N}{s}\left(\frac{s-2}{2}\right)}\|u\|^{2}_{{\mathcal H}_t},
$$
for all $s\in[2,\frac{2N}{N-2}]$
 if $N\geq3$ and for all $s\geq2$ if $N=2$, with $C_s>0$ as above.
\end{Lemma}
\begin{pf}
  Lemma \ref{l:ineqx2} implies that, if $v\in{\mathcal H}$,
  then $v\sqrt{G(\cdot,1)}\in H^1(\R^N)$ and the first embedding
  follows from  classical Sobolev
  inequalities and Lemma \ref{l:ineqx2}. The second inequality follows
  directly 
  from the first one and the change of variables $u(x)=v(x/\sqrt t)$. 
\end{pf}

\section{Spectrum of Ornstein-Uhlenbeck type operators with
  critical electromagnetic   potentials}\label{sec:spectr-analys-self}

In this section we describe the spectral properties of the operator
$L_{\A,a}= {\mathcal{L}}_{{\mathbf{A}},a}+\frac{x}2\cdot\nabla$ defined in (\ref{eq:13}). In particular we  extend to the
general critical electromagnetic case 
previous analogous results obtained in \cite{vazquez_zuazua} for
$\A\equiv 0$ and $a\equiv\lambda$ constant and in \cite{FP} for
$\A\equiv 0$.

In order to apply the Spectral Theorem to the operator $L_{\A,a}$,
some compactness is first needed. With this aim, following \cite{ES}, we prove the following compact embedding.

\goodbreak
\begin{Lemma}\label{l:compact}
The space ${\mathcal H}$ is compactly embedded in
${\mathcal L}$.
\end{Lemma}
\begin{pf}
Let us  assume that $u_{k}\weakly
u$ weakly in ${\mathcal H}$. From Rellich's theorem
$u_{k}\rightarrow u$ in  $L^{2}_{\rm loc}(\R^N)$, i.e. $u_k\to u$
strongly in $L^2(\Omega)$ for all $\Omega\subset\subset\R^N$. For every $R>0$ and $k\in\N$,
we have
\begin{equation}\label{eq:6}
\dint_{\R^N}|u_{k}-u|^2G(x,1)\,dx=
A_k(R)+B_k(R)
\end{equation}
where
\begin{equation}\label{eq:7}
A_k(R)=\int_{\{|x|\leq R\}}|u_{k}(x)-u(x)|^2e^{-|x|^2/{4}}\,dx
\to 0\quad\text{as }k\to+\infty,\quad \text{for every }R>0,
\end{equation}
and
$$
B_k(R)=
\int_{\{|x|>R\}}|u_{k}(x)-u(x)|^2G(x,1)\,dx.
$$
From Lemma \ref{l:ineqx2} and boundedness of $u_k$ in ${\mathcal H}$,
we deduce that
\begin{align}\label{eq:8}
&B_{k}(R)\leq
R^{-2}\dint_{\{|x|>R\}}|x|^2|u_{k}(x)-u(x)|^2G(x,1)\,dx\\
\notag&\leq
\frac{1}{R^2}\bigg(16\int_{\R^N}|\nabla_\A
(u_{k}-u)(x)|^2G(x,1)\,dx+4N\int_{\R^N}|u_{k}(x)-u(x)|^2G(x,1)\,dx\bigg)
\leq \frac{\rm const}{R^2}.
\end{align}
Combining (\ref{eq:6}), (\ref{eq:7}), and (\ref{eq:8}), we obtain that
 $u_{k}\rightarrow u$ strongly in ${\mathcal L}$.
\end{pf}

In the proof of the representation formula for solutions stated in
Theorem \ref{t:representation} and performed in
section \ref{sec:repr-form-solut} also the forward  Ornstein-Uhlenbeck
operators
$L^-_{\A,a}= {\mathcal{L}}_{{\mathbf{A}},a}-\frac{x}2\cdot\nabla$
will come into play. In order to study its spectral decomposition, we
introduce the forward analogue of the spaces $\mathcal H$ and
$\mathcal L$. More precisely, we define the space $\tilde{{\mathcal
    L}}$ as in \eqref{eq:49}
and $\tilde{\mathcal H}$ as
the completion of $C^{\infty}_{\rm c}(\R^N\setminus\{0\},\C)$ with respect to 
\begin{equation}\label{eq:62}
\|\varphi\|_{\tilde{\mathcal H} }=\bigg(\int_{{\mathbb{R}}^N}\bigg(|\nabla \varphi (x)|^2+
|\varphi (x)|^2+\frac{|\varphi (x)|^2}{|x|^2}\bigg) e^{\frac{|x|^2}4} \,dx\bigg)^{\!\!1/2}.
\end{equation}
\begin{Proposition}\label{p:iso}
Let $\tilde{T}: \mathcal H\to \tilde{\mathcal H}$ be defined as $\tilde{T} u(x)= e^{-\frac{|x|^2}{4}} u(x)$. Then,  
\begin{itemize}
\item [i)] $\tilde{T}:{\mathcal L}\to \tilde{{\mathcal L}}$ is an isometry;
\item [ii)] $\tilde{T}: \mathcal H\to \tilde{\mathcal H}$ is an
  isomorphism of normed vector spaces.
\end{itemize}
\end{Proposition}
\begin{pf}  We first observe that, if $\varphi=\tilde{T} u$, then 
\[
\dint_{\R^N} |\varphi(x)|^2 e^{\frac{|x|^2}{4}}\,dx =   \dint_{\R^N}
|u(x)|^2 e^{-\frac{|x|^2}{4}}\,dx.
\]
Hence, $\|\varphi\|_{\tilde{\mathcal L}} = \|u\|_{{\mathcal L}}=
\|\tilde Tu\|_{\tilde{\mathcal L}}$ for all $u\in\mathcal L$. Then i) is proved.

To prove ii), we first observe that, if $u\in C^\infty_{\rm
  c}(\R^N\setminus\{0\},\C)$ and $v=\tilde Tu$, there holds
\[
\nabla u(x)=e^{\frac{|x|^2}{4}}\left(\frac x2 v(x)+\nabla v(x)\right),
\]
and hence
\begin{align*}
|\nabla u(x)|^2&=e^{\frac{|x|^2}{2}}\left(\frac{|x|^2}4 |v(x)|^2+|\nabla
  v(x)|^2+\Re(\overline{v(x)}\nabla v(x)\cdot x)\right)\\
&=e^{\frac{|x|^2}{2}}\left(\frac{|x|^2}4 |v(x)|^2+|\nabla
  v(x)|^2+\frac12\nabla|v|^2(x)\cdot x\right).
\end{align*}
Then, an integration by parts yields
\begin{align*}
\int_{{\mathbb{R}}^N}|\nabla u(x)|^2&\, e^{-\frac{|x|^2}4} \,dx
=\int_{{\mathbb{R}}^N}e^{\frac{|x|^2}4}\left(\frac{|x|^2}4 |v(x)|^2+|\nabla
  v(x)|^2\right)\,dx+\frac12\int_{{\mathbb{R}}^N}e^{\frac{|x|^2}4}\nabla|v|^2(x)\cdot x \,dx\\
&=\int_{{\mathbb{R}}^N}e^{\frac{|x|^2}4}\left(\frac{|x|^2}4 |v(x)|^2+|\nabla
  v(x)|^2\right)\,dx-\frac12\int_{{\mathbb{R}}^N}e^{\frac{|x|^2}4}|v(x)|^2\bigg(\frac{|x|^2}2+N\bigg)\,dx\\
&=\int_{{\mathbb{R}}^N}e^{\frac{|x|^2}4}\left(|\nabla
  v(x)|^2-\frac N2 |v(x)|^2\right)\,dx.
\end{align*}
Therefore
\[
\int_{{\mathbb{R}}^N} e^{-\frac{|x|^2}4} \bigg(|\nabla u(x)|^2+
N|u(x)|^2+\frac{|u(x)|^2}{|x|^2}\bigg)dx=
\int_{{\mathbb{R}}^N}e^{\frac{|x|^2}4}\left(|\nabla
  v(x)|^2+\frac N2 |v(x)|^2+\frac{|v(x)|^2}{|x|^2}\right)dx.
\]
The above identity and density  of $C^\infty_{\rm c}(\R^N\setminus\{0\},\C)$
in ${\mathcal H}$ and  $\tilde{\mathcal H}$ prove statement ii).
\end{pf}

As a consequence, we
obtain the following compact embedding for~$\tilde{\mathcal H}$.
\begin{Corollary} The space $\tilde{\mathcal H}$ is compactly embedded in
${\tilde{\mathcal L}}$.
\end{Corollary}
\begin{pf}
The conclusion follows by combining Lemma \ref{l:compact} with Proposition \ref{p:iso}.
\end{pf}

The Hardy type inequalities established in section
\ref{sec:parabolic-hardy-type} and the embeddings discussed above
allow applying the classical Spectral Theorem to the operator
$L_{\A,a}$ defined in \eqref{eq:13}.
\begin{Lemma}\label{l:hilbasis}
  For $N\geq2$, let
  ${\mathbf{A}}\in C^1({\mathbb{S}} ^{N-1},{\mathbb{R}}^N)$ satisfy
  \eqref{transversality} and let 
  $a\in L^{\infty}\big({\mathbb S}^{N-1},\R\big)$ be such that
  \eqref{eq:hardycondition} holds. Then the spectrum of the operator
  $L_{\A,a}$ defined in (\ref{eq:13}) consists of a diverging sequence
  of real eigenvalues with finite multiplicity. Moreover, there exists
  an orthonormal basis of ${\mathcal L}$ whose elements belong to
  ${\mathcal H}$ and are eigenfunctions of $L_{\A,a}$.
\end{Lemma}
\begin{pf}
  By Corollary \ref{c:pos_def} and the Lax-Milgram Theorem, the bounded
  linear self-adjoint operator
$$
\mathcal T_{\A,a}:{\mathcal L}\to{\mathcal L},\quad
\mathcal T_{\A,a}=\bigg(L_{\A,a}+\frac{N-2}{4}\,{\rm Id}\bigg)^{-1}
$$
is well defined. Moreover, by Lemma \ref{l:compact}, $\mathcal T_{\A,a}$ is compact.
The result then follows from the Spectral Theorem.
\end{pf}

We are now in position to prove Proposition \ref{p:explicit_spectrum}.

\begin{pfn}{Proposition \ref{p:explicit_spectrum}}
 Let $\gamma$ be an eigenvalue of $L_{\A,a}$ and $g\in{\mathcal
    H}\setminus\{0\}$ be a corresponding eigenfunction, so that
\begin{equation}\label{eq:50}
\mathcal{L}_{\A,a} g(x)+ \frac{\nabla g(x)\cdot
  x}{2}=\gamma\, g(x)
\end{equation}
in a weak ${\mathcal H}$-sense.  From classical regularity theory for
elliptic equations, $g\in C^{1,\alpha}_{\rm loc}(\R^N\setminus\{0\},\C)$.
 Hence $g$ can be
expanded as
\begin{equation*}
g(x)=g(r\theta)=\sum_{k=1}^\infty\phi_k(r)\psi_k(\theta)
\quad \text{in }L^2({\mathbb S}^{N-1},\C),
\end{equation*}
where $r=|x|\in(0,+\infty)$, $\theta=x/|x|\in{{\mathbb S}^{N-1}}$, and
\begin{equation*}
  \phi_k(r)=\int_{{\mathbb S}^{N-1}}g(r\theta)
  \overline{\psi_k(\theta)}\,dS(\theta),
\end{equation*}
with $\{\psi_k\}_k$ as in (\ref{angular}).
Equations (\ref{angular}) and (\ref{eq:50}) imply that, for every $k$,
\begin{equation}\label{eq:51}
  \phi''_{k}+\left(\dfrac{N-1}{r}-\dfrac{r}{2}\right)
  \phi'_{k}+\left(\gamma-\dfrac{\mu_k(\A,a)}{r^2}\right)\phi_{k}=0
\quad\text{in  }(0,+\infty).
\end{equation}
The rest of the proof follows exactly as in \cite[Proposition 1]{FP} .
\end{pfn}

Denoting by $L_{m}^{\alpha}(t)$ the generalized
Laguerre polynomials
\begin{equation*}
L_{m}^{\alpha}(t)=\sum_{n=0}^m (-1)^{n}{\binom{m+\alpha}{m-n}\,\frac{t^n}{n!}%
},
\end{equation*}
we have that 
\begin{equation}\label{eq:15p}
P_{k,m}\bigg(\frac{|x|^2}{4}\bigg)=\binom{m+\beta_k}{m}%
^{\!\!-1} L_{m}^{\beta_k}\Big(\frac{|x|^2}{4}\Big),
\end{equation}
with $\beta_{k}=\sqrt{\big(\frac{N-2}{2}\big)^{2}+\mu_k({\mathbf{A}},a)}$.
It is worth recalling the well known orthogonality relation
\begin{equation}\label{eq:14p}
\int_{0}^{\infty} x^{\alpha} e^{-x} L_{n}^{\alpha}(x) L_{m}^{\alpha}(x)\,dx=%
\frac{\Gamma(n+\alpha+1)}{n!}\delta_{n,m},
\end{equation}
where $\delta_{n,m}$ denotes the Kronecker delta.

\begin{remark}\label{rem:ortho}
 For $n,j\in\N$, $j\geq1$, let $V_{n,j}$ be defined in \eqref{eq:66}.
 From the orthogonality of eigenfunctions $\{\psi_k\}_k$ in
  $L^2({\mathbb S}^{N-1},\C)$ and the orthogonality relation for
  Laguerre polynomials \eqref{eq:14p}, it follows that
$$
\text{if }(m_1,k_1)\neq(m_2,k_2)\quad\text{then}\quad
V_{m_1,k_1}\text{ and } V_{m_2,k_2}\text{ are orthogonal in }{\mathcal
  L}.
$$
By Lemma \ref{l:hilbasis}, we conclude that an orthonormal basis of
${\mathcal L}$ is given by 
$$
\left\{
\widetilde V_{n,j}=
\frac{V_{n,j}}{\|V_{n,j}\|_{\mathcal L}}: j,n\in\N,j\geq
  1\right\}.
$$
\end{remark}

\begin{remark}\label{rem:tilde}
If $\gamma$ is an eigenvalue and $\varphi\in \tilde{\mathcal
   H}\setminus\{0\}$ is a corresponding eigenfunction of the
 operator $L_{\A,a}^{-}$ defined as 
\begin{equation}\label{eq:Lmeno}
L^{-}_{{\mathbf{A}},a}:\tilde{\mathcal H}\to {
\tilde{\mathcal H}}^\star ,\quad L_{\A,a}^{-}\varphi(x)=\mathcal{L}_{\A,a} \varphi(x)-\frac{\nabla \varphi(x)\cdot
  x}{2},
\end{equation}
i.e. $L_{\A,a}^{-}\varphi=\gamma\, \varphi$ 
in a weak $\tilde{\mathcal H}$-sense, then $\gamma-\frac N2$ is an
eigenvalue of the operator $L_{\A,a}$ with $u=e^{\frac{|x|^2}{4}}
\varphi$ as a corresponding eigenfunction, i.e. 
$$
L_{\A,a} u(x)=\mathcal{L}_{\A,a} u(x)+\frac{\nabla u(x)\cdot
  x}{2}=\Big(\gamma-\frac{N}{2}\Big)\, u(x)
$$
in a weak ${\mathcal  H}$-sense. It follows that the set of the
eigenvalues
  of  $L^-_{\A,a}$ is
$$
\bigg\{ \frac N2+m-\frac{\alpha_k}2: k,m\in\N, k\geq 1\bigg\}.
$$
Moreover, letting $U_{n,j}= e^{-\frac{|x|^2}{4}} V_{n,j}$ with
$V_{n,j}$ as in \eqref{eq:66}, we have that
$$
\left\{
\widetilde U_{n,j}=
\frac{U_{n,j}}{\|U_{n,j}\|_{\tilde{\mathcal L}}}: j,n\in\N,j\geq
  1\right\}
$$
is an orthonormal basis of ${\tilde { \mathcal L}}$.
\end{remark}

\begin{remark}\label{rem_normalag} 
From \eqref{eq:15p} and the orthogonality relation \eqref{eq:14p}, it is easy to check that
\begin{align*}
  \|U_{m,k}\|_{\tilde{\mathcal L}}^{2}&=
\dint_{{\mathbb{R}}^N}e^{\frac{|x|^2}{4}}|U_{m,k}|^2\,dx=\dint_{{\mathbb{R}}^N} 
 e^{-\frac{|x|^2}{4}}|V_{m,k}|^2\,dx\\
&= \|V_{m,k}\|_{{\mathcal L}}^{2} =2^{1+2\beta_k}\Gamma(1+\beta_k)\binom{%
  m+\beta_k}{m}^{\!\!-1}.
\end{align*}
\end{remark}

\section{The parabolic electromagnetic Almgren monotonicity formula }\label{sec:almgren}

Throughout this section, we assume that $N\geq2$,
${\mathbf{A}}\in C^1({\mathbb{S}} ^{N-1},{\mathbb{R}}^N)$ and
$a\in L^{\infty}\big({\mathbb S}^{N-1},\R\big)$ satisfy
\eqref{transversality} and \eqref{eq:hardycondition}, and $\tilde u$
is a weak solution to (\ref{probtilde}) in $\R^N\times(0,T)$ with $h$
satisfying \eqref{eq:der}--\eqref{eq:h} in $I=(-T,0)$.  Let $C_1>0$
and $\overline T>0$ be as in Corollary \ref{c:pos_per} and denote
\begin{equation*}
  \alpha=\frac{T}{2\big(\big\lfloor{T}/{\overline T}\big\rfloor+1\big)},
\end{equation*}
where $\lfloor \cdot\rfloor$ denotes the floor function, i.e. $\lfloor
x\rfloor:=\max\{n\in\Z:\ n\leq x\}$.
Then
$(0,T)=\bigcup_{j=1}^k(a_j,b_j)$ being 
$$
k=2\big(\big\lfloor{T}/{\overline T}\big\rfloor+1\big)-1,
\quad
a_j=(j-1)\alpha,\quad\text{and}\quad
b_j=(j+1)\alpha.
$$
We notice that $0<2\alpha<{\overline T}$ and $(a_j,b_j)\cap(a_{j+1},b_{j+1})=
(j\alpha,(j+1)\alpha)\not =\emptyset$.
For every $j=1,\dots,k$, we define
\begin{equation*}
\tilde u_{j}(x,t)=\tilde u(x, t+a_j),\quad
x\in\R^N,\ t\in(0,2\alpha).
\end{equation*}

\begin{Lemma}\label{l:u_i}
  For every $j=1,\dots,k$, the function $\tilde u_j$ defined above
  is a weak solution to
\begin{equation}\label{prob_i}
-(\tilde{u}_j)_t(x,t)+{\mathcal{L}}_{{\mathbf{A}},a} \tilde{u}_j (x,t) = h(x,-(t+a_j)) \tilde{u}_j(x,t)
\end{equation}
  in $\R^N\times (0,2\alpha)$ in
  the sense of Definition \ref{def:solution}.  Furthermore,
 $\tilde v_j(x,t):=\tilde u_j(\sqrt{t}x,t)$~is a weak solution to
\begin{equation}\label{eq:eqforv_i}
  (\tilde v_j)_t+\frac1t\bigg(-{\mathcal{L}}_{{\mathbf{A}},a}\tilde
  v_j-\frac x2\cdot \nabla \tilde v_j+th(\sqrt{t} x,-(t+a_j))\tilde v_j(x,t)\bigg)=0
\end{equation}
 in $\R^N\times(0,2\alpha)$ in the sense of Remark
  \ref{rem:v2}.
\end{Lemma}
\begin{proof} 
Since the proof is very similar to the proof of \cite[Lemma 4.1]{FP},
we omit it here, referring to \cite{FP} for details. 
\end{proof}

\noindent
For every $j=1,\dots,k$, we  define
\begin{equation}\label{eq:Hi(t)}
  H_j(t)=\int_{\R^N}|\tilde u_{j}(x,t)|^2\, G(x,t)\,dx,
  \quad\text{for every }t\in (0,2\alpha),
\end{equation}
and
\begin{equation}\label{eq:Di(t)}
  D_j(t)=\!\!\int_{\R^N}\!\!\bigg[|\nabla_\A \tilde u_j(x,t)|^2-
  \dfrac{a({x}/{|x|})}{|x|^2}|\tilde u_j(x,t)|^2-
h(x,-(t+a_j))|\tilde u_{j}(x,t)|^2\!\bigg]G(x,t)dx
\end{equation}
for a.e. $t\in (0,2\alpha)$.  From Lemma \ref{l:u_i}  and Remark
  \ref{rem:uv} it follows that, for every $1\leq j\leq k$,
 $H_j\in W^{1,1}_{\rm loc}(0,2\alpha)$ and
\begin{equation}\label{eq:10i}
  H'_j(t)=2\,\Re \left[\!\!\!\!{\phantom{\bigg\langle}}_{{\mathcal
      H}_t^\star}\bigg\langle
  (\tilde u_j)_t+\frac{\nabla \tilde u_j\cdot x}{2t},\tilde u_j(\cdot,t)
  \bigg\rangle_{{\mathcal H}_t}\right]=2D_j(t)\quad\text{for a.e. }t\in(0,2\alpha).
\end{equation}

\begin{Lemma}\label{l:Hcreas}
  Letting $C_1$  as in Corollary \ref{c:pos_per}, we have that, for every $j=1,\dots,k$, the
  function 
$$
t\mapsto
 t^{-2C_1+\frac{N-2}{2}}H_j(t)$$
 is
nondecreasing in $(0,2\alpha)$.
\end{Lemma}
\begin{proof}
  From \eqref{eq:10i}, Corollary \ref{c:pos_per} and the fact that $2\alpha<{\overline
    T}$, we have that, for all $t\in(0,2\alpha)$,
$$
H'_{j}(t)\geq \frac1t\bigg(2C_1-\frac{N-2}2\bigg)H_{j}(t),
$$
which implies that
$\frac{d}{dt}\big(t^{-2C_1+\frac{N-2}{2}}H_{j}(t)\big)\geq 0$, thus
concluding the proof.
\end{proof}

\noindent
\begin{Lemma}\label{l:Hpos}
If  $1\leq j\leq k$
and  $H_j(\bar t)=0$
  for some $\bar t\in(0,2\alpha)$, then $H_{j}(t)=0$ for all $t\in
  (0,\bar t\,]$.
\end{Lemma}
\begin{proof}
  From Lemma \ref{l:Hcreas}, the function $t\mapsto
  t^{-2C_1+\frac{N-2}{2}}H_{j}(t)$ is nondecreasing in $(0,2\alpha)$,
  nonnegative, and vanishing at $\bar t$. Hence $H_j(t)=0$
  for all $t\in (0,\bar t]$.
\end{proof}

\begin{Lemma}\label{l:Dprime}
  If $1\leq j\leq k$ and $T_j\in(0,2\alpha)$ is such that
  $\tilde u_j(\cdot,T_j)\in{\mathcal H}_{T_j}$, then
\begin{itemize}
\item[(i)]
$\int_\tau^{T_j}\int_{\R^N}\big(\big|(\tilde u_j)_t(x,t)+\frac{\nabla_\A
      \tilde u_j(x,t)\cdot x}{2t}\big|^2G(x,t)\,dx\big)\,dt<+\infty
\quad\text{for all
  }\tau\in (0,T_j)$;\\[1pt]
\item[(ii)]
the function $t\mapsto t D_j(t)$ belongs to $W^{1,1}_{\rm loc}(0,T_j)$
and, for a.e. $t\in(0,T_j)$, 
\end{itemize}
\begin{align*}
&  \frac{d}{dt}\,
  \big(tD_j(t)\big)=2t\int_{\R^N}\bigg|(\tilde u_j)_t(x,t)+\frac{\nabla
      \tilde u_j(x,t)\cdot x}{2t}\bigg|^2G(x,t)\,dx\\
 &\quad + \!\int_{\R^N}\!h(x,-(t+a_j))\bigg(\frac{N-2}{2}|\tilde
  u_j(x,t)|^2+
\Re(\tilde u_j(x,s)\overline{\nabla \tilde u_j(x,s)\cdot
                x})-\frac{|x|^2}{4t}|\tilde u_j(x,t)|^2\bigg)
  G(x,t)\,dx\\
  &\quad+t\int_{\R^N}h_{t}(x,-(t+a_j))|\tilde u_j(x,t)|^2G(x,t)\,dx.
\end{align*}
\end{Lemma}
\begin{proof} Testing equation (\ref{eq:eqforv_i})
  with $(\tilde v_j)_t$ (this formal testing
  procedure can be made rigorous by a suitable approximation) and
  using  Corollary \ref{c:pos_per}, we obtain that, for all
  $t\in(0,T_j)$,
\begin{align*}
  \int_t^{T_j}s&\bigg(\int_{\R^N}|(\tilde v_j)_t(x,s)|^2G(x,1)\,dx\bigg)\,ds\\
&\leq{\rm
    const\,}\bigg( \|\tilde u_j(\sqrt {T_j}\,\cdot,T_j)\|^2_{\mathcal H}+
  \int_{\R^N}|\tilde v_j(x,t)|^2G(x,1)\,dx\\
  & \quad+\int_t^{T_j}\!\!\bigg(\int_{\R^N}\!\!h(\sqrt s x, -(s+a_j))\bigg(
  \frac{|x|^2}8|\tilde v_j(x,s)|^2\\
&\hskip2.5cm- \frac{\Re(\tilde v_j(x,s)\overline{\nabla \tilde v_j(x,s)\cdot x}) }{2}
  -\frac{N-2}{4}|\tilde v_j(x,s)|^2\bigg)G(x,1)\,dx\bigg)\,ds\\
  &\quad-\frac{1}{2}\int_t^{T_j}\!\!s\bigg(\int_{\R^N}h_{s}(\sqrt s x,
  -(s+a_i))|\tilde v_j(x,s)|^2G(x,1)\,dx\bigg)\,ds\bigg).
\end{align*}
From \eqref{eq:der}--\eqref{eq:h} and Lemmas
\ref{l:ineqx2} and \ref{l:sob} we have that the integrals in the last two
terms of the previous formula are finite for every $t\in(0,T_j)$. Hence we
conclude that
$$
(\tilde v_j)_t\in L^2(\tau,T_j;{\mathcal L})\quad\text{for all
  }\tau\in (0,T_j).
$$
Testing \eqref{eq:eqforv_i} with $(\tilde v_j)_t$ also yields
\begin{align*}
  &\int_t^{T_i}s\bigg(\int_{\R^N}|(\tilde v_j)_t(x,s)|^2G(x,1)\,dx\bigg)\,ds\\
  &\quad+ \frac12 \int_{\R^N}\!\!\bigg(|\nabla_\A
  \tilde v_j(x,t)|^2-\frac{a(x/|x|)}{|x|^2}\,|\tilde v_j(x,t)|^2-th(\sqrt t x,
  -(t+a_j))|\tilde v_j(x,t)|^2\bigg)
  G(x,1)\,dx\\
  &=\frac12 \int_{\R^N}\!\!\bigg(|\nabla_\A
  v_{0,j}(x)|^2-\frac{a(x/|x|)}{|x|^2}\,|v_{0,j}(x)|^2- T_j h(\sqrt {
    T_j} x, -(T_j+a_j))|v_{0,j}(x)|^2\bigg) G(x,1)\,dx
  \\
  &\quad+\int_t^{T_j}\!\!\bigg(\int_{\R^N}h(\sqrt s x, -(s+a_ij))\bigg(
  \frac{|x|^2}8|\tilde v_j(x,s)|^2\\
&\hskip2.5cm- \frac{\Re(\tilde v_j(x,s)\overline{\nabla \tilde v_j(x,s)\cdot x}) }{2}
  -\frac{N-2}{4}|\tilde v_j(x,s)|^2\bigg)G(x,1)\,dx\bigg)\,ds\\
  &\quad-\frac{1}{2}\int_t^{T_j}s\bigg(\int_{\R^N}h_{s}(\sqrt s x,
  -(s+a_j))|\tilde v_j(x,s)|^2G(x,1)\,dx\bigg)\,ds,
\end{align*}
for all $t\in (0, T_j)$, where $v_{0,j}(x):=\tilde u_j(\sqrt { T_j} x, T_j)
\in{\mathcal H}$.  Therefore the function
$$
t\mapsto 
\int_{\R^N}\!\!\bigg(|\nabla_\A
  \tilde v_j(x,t)|^2-\frac{a(x/|x|)}{|x|^2}\,|\tilde v_j(x,t)|^2-th(\sqrt t x,
  -(t+a_j))|\tilde v_j(x,t)|^2\bigg)
  G(x,1)\,dx
$$
is absolutely continuous  in $(\tau,T_j)$ for all $\tau\in(0,T_j)$ and
\begin{align*}
  \frac{d}{dt}&
\int_{\R^N}\!\!\bigg(|\nabla_\A
  \tilde v_j(x,t)|^2-\frac{a(x/|x|)}{|x|^2}\,|\tilde v_j(x,t)|^2-th(\sqrt t x,
  -(t+a_j))|\tilde v_j(x,t)|^2\bigg)
  G(x,1)\,dx\\
  &=2t \int_{\R^N}|(\tilde v_j)_t(x,t)|^2G(x,1)\,dx\\
  &\quad-\int_{\R^N}h(\sqrt t x, -(t+a_j))\bigg( \frac{|x|^2}4|\tilde v_j(x,t)|^2\\
&\hskip2.5cm-
\Re(\tilde v_j(x,s)\overline{\nabla \tilde v_j(x,s)\cdot x}) 
  -\frac{N-2}{2}|\tilde v_j(x,t)|^2\bigg)G(x,1)\,dx\\
  &\quad+t\int_{\R^N}h_{t}(\sqrt t x, -(t+a_j))|\tilde v_j(x,t)|^2G(x,1)\,dx.
\end{align*}
The change of variables $\tilde u_j(x,t)=\tilde v_j(x/\sqrt t,t)$
gives the conclusion.
\end{proof}

For all $j=1,\dots,k$, the \emph{Almgren type
  frequency function} associated to $\tilde u_j$ is defined as 
\begin{equation*}
N_j:(0,2\alpha)\to\R\cup\{-\infty,+\infty\},
\quad N_j(t):=\frac{tD_j(t)}{H_j(t)}.
\end{equation*}
The analysis below will show that each $N_j$
actually assumes
 finite values all over $(0,2\alpha)$
and  its derivative 
is an integrable
perturbation of a nonnegative function wherever  $N_j$ assumes
finite values.

\begin{Lemma}\label{l:Nprime}
  Let $k\in \{1,\dots,k\}$. If there exist $\beta_j,T_j\in
  (0,2\alpha)$ such that
\begin{equation}\label{eq:10j}
\beta_j<T_j, \quad H_j(t)>0 \text{ for all $t\in (\beta_j,T_j)$},\quad
\text{and}\quad \tilde u_j(\cdot,T_j)\in {\mathcal H}_{T_j},
\end{equation}
then $N_j\in W^{1,1}_{\rm loc}(\beta_j,T_j)$ and
\begin{align*}
N'_j(t)={\nu}_{1j}(t)+{\nu}_{2j}(t)
\end{align*}
in a distributional sense and a.e. in $(\beta_j,T_j)$ where
\begin{align*}
  {\nu}_{1j}(t)&=\frac{2t}{H_{j}^2(t)}{{
      \bigg[\bigg(\!\int_{\R^N}\!\!\bigg|(\tilde
                 u_j)_t(x,t)+\frac{\nabla \tilde u_j(x,t)\cdot
        x}{2t}\bigg|^2\!G(x,t)\,dx\!\bigg)
      \bigg(\!\int_{\R^N}\!\!|\tilde u_{j}(x,t)|^2\, G(x,t)\,dx\!\bigg)}}\\
  &\hskip3cm{{-\bigg(\int_{\R^N}\Big((\tilde u_j)_t(x,t)+\frac{\nabla
        \tilde u_{j}(x,t)\cdot x}{2t}\Big)\overline{\tilde u_{j}(x,t)}G(x,t)\,dx
      \bigg)^{\!\!2}\,\bigg]}}
\end{align*}
and
\begin{align*} 
{\nu}_{2j}(t)&=\frac1{H_{j}(t)}
      \int_{\R^N}h(x,-(t+a_j))\bigg(\frac{N-2}{2}|\tilde u_{j}(x,t)|^2\\
&\hskip3cm+
  \Re\big(\tilde u_j(x,t) \overline{\nabla
       \tilde  u_{j}(x,t)\cdot x}\big)-\frac{|x|^2}{4t}|\tilde u_{j}(x,t)|^2\bigg)
      G(x,t)\,dx\\
  &\quad{+\dfrac
      t{H_{j}(t)}\bigg(\int_{\R^N}h_{t}(x,-(t+a_j))|\tilde u_{j}(x,t)|^2G(x,t)\,dx\bigg)}.
\end{align*}
\end{Lemma}
\begin{proof}
  From \eqref{eq:10i}  and Lemma \ref{l:Dprime}, it follows that
  $N_{j}\in W^{1,1}_{\rm loc}(\beta_j,T_j)$. From \eqref{eq:10i} it follows that
$$
N'_{j}(t)=\frac{(tD_{j}(t))'H_{j}(t)-tD_{j}(t)H'_{j}(t)}{H_{j}^2(t)}=
\frac{(tD_{j}(t))'H_{j}(t)-2tD_{j}^2(t)}{H_{j}^2(t)},
$$
and hence, in view of (\ref{eq:Hi(t)}),
\eqref{eq:10i}, and Lemma \ref{l:Dprime}, we obtain the conclusion.
\end{proof}

\begin{Lemma}\label{l:Nabove}
There exists $C_3>0$ such that, if $j\in \{1,\dots,k\}$ and $\beta_j,T_j\in
  (0,2\alpha)$ satisfy \eqref{eq:10j}, then,
 for every $t\in(\beta_j,T_j)$,
\begin{equation*}
  N_j(t)\leq
-\frac{N-2}{4}+C_{3}\bigg(N_j(T_j)+\frac{N-2}4\bigg).
\end{equation*}
\end{Lemma}

\begin{proof}
Let us first claim that, for all $j=1,\dots,k$, the  term ${\nu_{2j}}$ of Lemma \ref{l:Nprime}  can be estimated as follows:
\begin{equation}\label{eq:est_nu2}
\big|{\nu_{2j}}(t)\big|\leq 
 \begin{cases}
    C_4\big(N_j(t)+{\textstyle{\frac{N-2}{4}}}\big)
\Big(t^{-1+\e/2}+\|h_t(\cdot,-(t+a_j))\|_{L^{{N}/{2}}(\R^N)}\Big),&\text{if }N\geq3,\\
    C_4\big(N_j(t)+{\textstyle{\frac{N-2}{4}}}\big)
\Big(t^{-1+\e/2}+\|h_t(\cdot,-(t+a_j))\|_{L^{p}(\R^N)}\Big),&\text{if }N=2,
\end{cases}
\end{equation}
for a.e. $t\in(\beta_j,T_j)$, with some $C_4>0$ independent of $t$ and
$j$. Indeed, from (\ref{eq:h}) it follows that,  for a.e. $t\in(\beta_j,T_j)$,
\begin{align}\label{eq:15}
  \bigg|\int_{\R^N}&h(x,-(t+a_j))
  \Re\big(\tilde u_j(x,t) \overline{\nabla
       \tilde  u_{j}(x,t)\cdot x}\big)
      G(x,t)\,dx\bigg|\\
  &\notag\leq
  C_h\int_{\R^N}(1+|x|^{-2+\e})|\nabla \tilde u_{j}(x,t)||x||\tilde u_{j}(x,t) |G(x,t)\,dx\\
   &\notag\leq C_ht \int_{\R^N}|\nabla
  \tilde u_{j}(x,t)|\frac{|x|}{t}|u_{i}(x,t)| G(x,t)\,dx \\
&\notag\hskip1cm +C_ht^{\e/2}
  \int_{\{|x|\leq \sqrt t\}}|\nabla \tilde u_{j}(x,t)|\frac{|\tilde u_{j}(x,t)|}{|x|} G(x,t)\,dx\\
   &\notag\hskip1cm +C_ht^{\e/2}\int_{\{|x|\geq \sqrt t\}}|\nabla
  \tilde u_{j}(x,t)|\frac{|x|}{t}|\tilde u_{j}(x,t)| G(x,t)\,dx
  \\
   &\notag \leq \frac12C_h(t+t^{\e/2})\int_{\R^N}|\nabla
  \tilde u_{j}(x,t)|^2G(x,t)\,dx\\
   &\notag\hskip1cm+
  \frac12C_h(t+t^{\e/2})\int_{\R^N}\frac{|x|^2}{t^2}|\tilde u_{j}(x,t)|^2 G(x,t)\,dx\\
   &\notag\hskip1cm + \frac12C_ht^{\e/2}\int_{\R^N}|\nabla
  \tilde u_{j}(x,t)|^2G(x,t)\,dx
  + \frac12C_ht^{\e/2}\int_{\R^N}\frac{|\tilde u_{j}(x,t)|^2}{|x|^2}G(x,t)\,dx\\
  &\notag\leq \frac12C_ht^{\e/2}(2+{\overline T}^{1-\e/2})
  \int_{\R^N}|\nabla \tilde u_{j}(x,t)|^2G(x,t)\,dx\\
   &\notag\hskip1cm +
  \frac12C_ht^{\e/2}(1+{\overline T}^{1-\e/2})\int_{\R^N}
  \frac{|x|^2}{t^2}|\tilde u_{j}(x,t)|^2 G(x,t)\,dx\\
   &\notag\hskip1cm  +
  \frac12C_ht^{\e/2}\int_{\R^N}\frac{|\tilde u_{j}(x,t)|^2}{|x|^2}G(x,t)\,dx ,
\end{align}
and
\begin{align}\label{eq:14}
  \int_{\R^N}&|h(x,-(t+a_j))||x|^2|\tilde u_{j}(x,t)|^2 G(x,t)\,dx\\
  &\notag\leq C_h
  \int_{\R^N}|x|^2|\tilde u_{j}(x,t)|^2G(x,t)\,dx+C_h
  \int_{\R^N}|x|^{-2+\e}|x|^2|\tilde u_{j}(x,t)|^2 G(x,t)\,dx\\
  &\notag\leq C_h \int_{\R^N}|x|^2|\tilde u_{j}(x,t)|^2 G(x,t)\,dx+C_h t^{\e/2}
  \int_{\{|x|\leq \sqrt t\}}|\tilde u_{j}(x,t)|^2 G(x,t)\,dx\\
  &\notag \quad+C_ht^{-1+\e/2}
  \int_{\{|x|\geq \sqrt t \}}|x|^2|\tilde u_{j}(x,t)|^2 G(x,t)\,dx\\
  &\notag\leq C_ht^{-1+\e/2}(1+{\overline T}^{1-\e/2})
  \int_{\R^N}|x|^2|\tilde u_{j}(x,t)|^2 G(x,t)\,dx +C_h t^{\e/2} \int_{\R^N}|\tilde u_{j}(x,t)|^2 G(x,t)\,dx.
\end{align}
Moreover, by H\"older's inequality and
Corollary \ref{l:sob}, we have that, for a.e. $t\in (\beta_j,T_j)$, if $N\geq3$,
\begin{align}\label{eq:hip}
&\bigg|\int_{\R^N}h_{t}(x,-(t+a_j))|\tilde
  u_{j}(x,t)|^2G(x,t)\,dx\bigg|\leq C_{2^*}t^{-1}
\|\tilde u_j(\cdot,t)\|^2_{{\mathcal H}_t}\|h_t(\cdot,-(t+a_j))\|_{L^{{N}/{2}}(\R^N)}
\end{align}
and, if $N=2$,
\begin{align}\label{eq:hip-2}
&\bigg|\int_{\R^N}h_{t}(x,-(t+a_j))|\tilde
  u_{j}(x,t)|^2G(x,t)\,dx\bigg|\leq C_{\frac{2p}{p-1}}t^{-1/p}
\|\tilde u_j(\cdot,t)\|^2_{{\mathcal H}_t}\|h_t(\cdot,-(t+a_j))\|_{L^{p}(\R^N)}.
\end{align}
 Collecting \eqref{eq:11}, \eqref{eq:15},
\eqref{eq:14} and \eqref{eq:hip}--\eqref{eq:hip-2}, we obtain that
\begin{align}\label{eq:16}
\big|{\nu_{2j}}(t)\big|\leq &\,\frac{{\rm const}\, t^{\e/2}}{H_{j}(t)}
\bigg(\frac1t\int_{\R^N}|\tilde u_{j}(x,t)|^2 G(x,t)\,dx+
\int_{\R^N}\frac{|\tilde u_{j}(x,t)|^2}{|x|^2}G(x,t)\,dx\\
&\notag+\int_{\R^N}|\nabla \tilde u_{j}(x,t)|^2G(x,t)\,dx+
\frac1{t^2}\int_{\R^N}|x|^2|\tilde u_{j}(x,t)|^2
  G(x,t)\,dx\bigg)+\mathcal R_j(t)
\end{align}
where 
\begin{equation*}
\mathcal R_j(t)=
\begin{cases}
\frac{C_{2^*}}{H_{j}(t)}
\|\tilde u_j\|^2_{{\mathcal
    H}_t}\|h_t(\cdot,-(t+a_j))\|_{L^{{N}/{2}}(\R^N)},&\text{if
}N\geq3,\\
C_{\frac{2p}{p-1}}\frac{t^{1-\frac 1p}}{H_{j}(t)}
\|\tilde u_j\|^2_{{\mathcal
    H}_t}\|h_t(\cdot,-(t+a_j))\|_{L^{p}(\R^N)},&\text{if
}N=2.
\end{cases}
\end{equation*}
Then estimate \eqref{eq:est_nu2}  follows from inequality (\ref{eq:16}), Corollary
\ref{c:pos_per}, and Corollary \ref{cor:ineq}.

By Schwarz's inequality,
\begin{equation}\label{eq:17}
{\nu}_{1j}\geq 0 \quad\text{a.e. in }(\beta_j,T_j).
\end{equation}
From Lemma \ref{l:Nprime}, (\ref{eq:17}), and \eqref{eq:est_nu2} it
follows that
\begin{equation*}
\frac{d}{dt}N_{j}(t)\geq
 \begin{cases}
-    C_4\big(N_j(t)+{\textstyle{\frac{N-2}{4}}}\big)
\Big(t^{-1+\e/2}+\|h_t(\cdot,-(t+a_j))\|_{L^{{N}/{2}}(\R^N)}\Big),&\text{if }N\geq3,\\
 -   C_4\big(N_j(t)+{\textstyle{\frac{N-2}{4}}}\big)
\Big(t^{-1+\e/2}+\|h_t(\cdot,-(t+a_j))\|_{L^{p}(\R^N)}\Big),&\text{if }N=2,
\end{cases}
\end{equation*}
for a.e. $t\in(\beta_j,T_j)$.  After integration, it follows that
\begin{equation*}
  N_{j}(t)\\
  \leq
\begin{cases}
-\frac{N-2}{4}
+\Big(
N_{j}(T_{j})+\frac{N-2}{4}\Big)\exp\Big({\frac{2C_4}{\e}T_{j}^{\e/2}+C_4
\|h_t\|_{L^{1}((-T,0),L^{{N}/{2}}(\R^N))}}\Big),
  &\text{if }N\geq 3,\\
-\frac{N-2}{4}
+\Big(
N_{j}(T_{j})+\frac{N-2}{4}\Big)\exp\Big({\frac{2C_4}{\e}T_{j}^{\e/2}+C_4
\|h_t\|_{L^{1}((-T,0),L^{p}(\R^N))}}\Big),
  &\text{if }N=2,
\end{cases}
\end{equation*}
for any $t\in (\beta_j,T_j)$, thus yielding the conclusion.
\end{proof}

As a consequence of the above monotonicity argument, we can prove the
following non-vanishing properties of the the functions $H_j$.

\begin{Proposition}\label{p:t_0=0}
\begin{enumerate}[\rm (i)]
\item Let $j\in\{1,\dots,k\}$. If $H_j\not\equiv 0$, then
$H_j(t)>0$ for all $t\in(0,2\alpha)$.
\item Let $j\in\{1,\dots,k\}$. Then
$H_j(t)\equiv 0$ in $(0,2\alpha)$ if and only if $H_{j+1}(t)\equiv 0$
in $(0,2\alpha)$.
\item     Let $u\not\equiv0$ is  a weak solution to
  \eqref{prob} in $\R^N\times (-T,0)$ with  $h$ satisfying
  \eqref{eq:der}--\eqref{eq:h} in  $\R^N\times(-T,0)$. Let $\tilde
  u(x,t)=u(x,-t)$ and let $H_j$ be defined as in \eqref{eq:Hi(t)} for $j=1,2,\dots,k$. Then
$H_j(t)>0$ 
for all $t\in(0,2\alpha)$ and $j=1,\dots,k$. In particular,
\begin{equation}\label{eq:88}
\int_{\R^N}| u(x,-t)|^2G(x,-t)\,dx>0\quad\text{for all }t\in(-T,0).
\end{equation}
\end{enumerate}
\end{Proposition}
\begin{pf}
Once Lemmas \ref{l:Hpos} and \ref{l:Nabove} are established, the proof
can be done arguing as in \cite[Lemmas 4.9 and 4.10, Corollary
7]{FP}. Hence we omit it.
\end{pf}

\begin{pfn}{Proposition \ref{p:uniq_cont}}
Up a translation in time, it is not restrictive to assume that
$I=(-T,0)$ for some $T>0$. Then the conclusion follows from
Proposition \ref{p:t_0=0} (iii).
\end{pfn}

Henceforward, we assume $u\not\equiv 0$ in $\R^N\times(-T,0)$ (so that 
$\tilde u\not\equiv 0$ in $\R^N\times(0,T)$)
and we denote, for all
$t\in(0,2\alpha)$,
\begin{align*}
&H(t)=H_{1}(t)=\int_{\R^N}|\tilde u(x,t)|^2\, G(x,t)\,dx,\\
& D(t)=D_{1}(t)=\!\!\int_{\R^N}\!\!\bigg(|\nabla_\A \tilde u(x,t)|^2-
  \dfrac{a\big({x}/{|x|}\big)}{|x|^2}|\tilde u(x,t)|^2- h(x,-t)|\tilde u(x,t)|^2\bigg)\, G(x,t)\,dx.
\end{align*}
Proposition \ref{p:t_0=0} ensures that, if $\tilde u\not\equiv 0$ in
$\R^N\times (0,T)$, then $H(t)>0$ for all $t\in(0,2\alpha)$. Hence the
\emph{Almgren type frequency function}
\begin{equation*} 
  {\mathcal N}(t)=N_1(t)=\frac{tD(t)}{H(t)}
\end{equation*}
is well defined over all $(0,2\alpha)$. Moreover, by Lemma \ref{l:Nprime},
${\mathcal N}\in W^{1,1}_{\rm loc}(0,2\alpha)$ and
${\mathcal N}'={\nu}_1+{\nu}_2$  a.e. in $(0,2\alpha)$, 
where
${\nu}_1={\nu}_{11}$ and ${\nu}_2={\nu}_{21}$,
with
${\nu}_{11},{\nu}_{21}$ as in Lemma \ref{l:Nprime}.
By (\ref{eq:defsol1}),
 $\tilde u(\cdot,t)\in
{\mathcal H}_t$ for a.e. $t\in (0,T)$, hence there exists a $T_0\in (0,2\alpha)$ such that
$\tilde u(\cdot,T_0)\in {\mathcal H}_{T_0}$.
We now prove the existence of the limit of ${\mathcal N}(t)$ as $t\to 0^+$.

\begin{Lemma}\label{l:limit}
The limit
$\gamma:=\lim_{t\to 0^+}{\mathcal N}(t)$
exists and it is finite.
\end{Lemma}
\begin{pf}
We first observe that, in view of Corollary \ref{c:pos_per},
$tD(t)\geq \big(C_1-\frac{N-2}4\big)H(t)$  for all $t\in(0,2\alpha)$. Hence
${\mathcal N}(t)\geq  C_1-\frac{N-2}4$, i.e. $\mathcal N$ is bounded
from below. 
Let $T_0$ be as above.  By Schwarz's inequality,
${\nu}_1(t)\geq 0$ for a.e. $t\in (0,T_0)$. Furthermore, Lemma
\ref{l:Nabove} and estimate \eqref{eq:est_nu2}, together with
assumption \eqref{eq:der}, imply that ${\nu}_2\in L^{1}(0,T_0)$. 
Hence ${\mathcal N}'(t)$ is  the
sum of a nonnegative function and of a $L^1$ function over $(0,T_0)$.
Therefore, ${\mathcal N}(t)={\mathcal N}(T_0)-\int_{t}^{T_0}{\mathcal N}'(s)\,ds$
admits a limit as $t\rightarrow 0^{+}$. We conclude by observing that
such a limit is finite since $\mathcal N$ is bounded from below (as
observed at the beginning of the proof) and from above (due to Lemma
\ref{l:Nabove}) in the interval $(0,T_0)$.
\end{pf}

\begin{Lemma}\label{stimaH}
  Let $\gamma:=\lim_{t\rightarrow 0^+} {\mathcal N}(t)$ be as in Lemma
  \ref{l:limit}.  Then there exists a constant $K_1>0$ such that
\begin{equation}\label{eq:52}
H(t)\leq K_1 t^{2\gamma}  \quad \text{for all } t\in (0,T_0).
\end{equation}
Furthermore, for any $\sigma>0$, there exists a constant
$K_2(\sigma)>0$ depending on $\sigma$ such that
\begin{equation}\label{eq:53}
  H(t)\geq K_2(\sigma)\, t^{2\gamma+\sigma}\quad \text{for all } t\in (0,T_0).
\end{equation}
\end{Lemma}
\begin{pf}
 From Lemma \ref{l:Nprime},  Schwarz's inequality, 
 \eqref{eq:est_nu2}, and assumption \eqref{eq:der} we deduce that 
\begin{align*}
  {\mathcal N}(t)-\gamma&=\int_0^t ({\nu}_1(s)+{\nu}_2(s))ds \geq
  \int_0^t {\nu}_2(s)ds\\[5pt]
&\geq
\begin{cases}
-  C_3C_4\big({\mathcal N}(T_0)+{\textstyle{\frac{N-2}{4}}}\big)
\int_0^t
\Big(s^{-1+\e/2}+\|h_t(\cdot,-s)\|_{L^{N/2}(\R^N)}\Big)\,ds,
  &\text{if }N\geq3,\\[5pt]
-  C_3C_4{\mathcal N}(T_0)
\int_0^t
\Big(s^{-1+\e/2}+\|h_t(\cdot,-s)\|_{L^{p}(\R^N)}\Big)\,ds,
  &\text{if }N=2,
\end{cases}\\[5pt]
&\geq
\begin{cases}
-  C_3C_4\big({\mathcal N}(T_0)+{\textstyle{\frac{N-2}{4}}}\big)
\Big(\frac2\e t^{\e/2}+\|h_t\|_{L^{r}((-T,0),L^{{N}/{2}}(\R^N))}t^{1-1/r}\Big),
  &\text{if }N\geq3,\\[5pt]
-   C_3C_4{\mathcal N}(T_0) \Big(\frac2\e t^{\e/2}+\|h_t\|_{L^{r}((-T,0),L^{p}(\R^N))}t^{1-1/r}\Big),
  &\text{if }N=2
\end{cases}
\\[5pt]
&\geq - C_5t^\delta
\end{align*}
with 
\begin{equation}\label{eq:delta}
\delta=\min\left\{\frac\e2,1-\frac 1r\right\}
\end{equation}
for some constant $C_5>0$ and for all $t\in(0,T_0)$.
 From the above estimate and 
\eqref{eq:10i}, we deduce that
$( \log H(t))'=\frac{H'(t)}{H(t)}=\frac{2}{t}{\mathcal N}(t)\geq
\frac2t\gamma-2 C_5t^{-1+\delta}$, which, after integration over $(t,
T_0)$, yields \eqref{eq:52}.

Since $\gamma=\lim_{t\rightarrow 0^+}
{\mathcal N}(t)$, for any $\sigma>0$ there exists $t_\sigma>0$ such
that ${\mathcal N}(t)<\gamma+\sigma/2$ for any $t\in
(0,t_\sigma)$. Hence 
$\frac{H'(t)}{H(t)}=\frac{2\,{\mathcal
    N}(t)}{t}<\frac{2\gamma+\sigma}{t}$ 
which, by integration over $(t,t_\sigma)$ and continuity of $H$
outside $0$, implies \eqref{eq:53}  for some constant $K_2(\sigma)$
depending on $\sigma$.
\end{pf}

\section{Blow-up analysis}\label{sec:blow-up-analysis}

Once the monotonicity type Lemma \ref{l:limit} is established, our
next step is a blow-up analysis for scaled solutions to
(\ref{probtilde}), which can be performed following the procedure
developed in \cite{FP} 

Let $\tilde u$ be a weak solution to
  (\ref{probtilde}) in $\R^N\times(0,T)$ with  $h$ satisfying
  \eqref{eq:der}--\eqref{eq:h} in $I=(-T,0)$.
For every $\lambda>0$,
we define 
\begin{equation*}
\tilde u_{\lambda}(x,t)=\tilde u(\lambda x,\lambda^2t).
\end{equation*}
We observe that $\tilde u_{\lambda}$ weakly solves
\begin{equation}\label{lambda}
-(\tilde u_{\lambda})_t+\mathcal L_{\A,a}\tilde u_\lambda=\lambda^2
h(\lambda x, -\lambda^2 t)\tilde u_{\lambda}\quad\text{in }\R^N\times (0,T/\lambda^2).
\end{equation}
We can associate to the  scaled
equation (\ref{lambda}) the Almgren frequency function 
\begin{equation}\label{eq:34}
{\mathcal N}_\lambda(t)=\frac{t\,D_\lambda(t)}{H_\lambda(t)},
\end{equation}
where
\begin{align*}
  D_{\lambda}(t)&=
                  \int_{\R^N}\!\!\bigg(|\nabla_\A \tilde u_\lambda(x,t)|^2-
                  \dfrac{a\big({x}/{|x|}\big)}{|x|^2}|\tilde
                  u_\lambda(x,t)|^2-\lambda^2 h(\lambda x,-\lambda^2
                  t)|\tilde u_\lambda(x,t)|^2\bigg)\, G(x,t)\,dx,\\
                  H_{\lambda}(t)&=\int_{\R^N}|\tilde u_{\lambda}(x,t)|^2 G(x,t)\,dx.
\end{align*}
By scaling and a suitable change
of variables we easily see that
\begin{align}\label{eq:19}
D_\lambda(t)=\lambda^2D(\lambda^2t)\quad\text{and}
\quad H_\lambda(t)=H(\lambda^2t),
\end{align}
so that
\begin{align}\label{eq:scale_for_N}
{\mathcal N}_\lambda(t)={\mathcal N}(\lambda^2t)\quad
\text{for all }t\in\Big(0,{\frac{2\alpha}{\lambda^2}}\Big).
\end{align}

\begin{Lemma}\label{l:blow_up}
  Let $N\geq2$,
  ${\mathbf{A}}\in C^1({\mathbb{S}} ^{N-1},{\mathbb{R}}^N)$ and let
  $a\in L^{\infty}\big({\mathbb S}^{N-1},\R\big)$ satisfy
  \eqref{transversality} and \eqref{eq:hardycondition}. Let
  $\tilde u\not\equiv 0$ be a nontrivial weak solution to
  (\ref{probtilde}) in $\R^N\times(0,T)$ with $h$ satisfying
  \eqref{eq:der}--\eqref{eq:h} in $I=(-T,0)$. Let
  $\gamma:=\lim_{t\to 0^+}{\mathcal N}(t)$ as in Lemma \ref{l:limit}.
  Then
\begin{itemize}
\item[(i)] $\gamma$ is an eigenvalue of the operator $L_{\A,a}$ defined in
  (\ref{eq:13});
\item[(ii)] for every  sequence $\lambda_n\to 0^+$,
there exists a subsequence $\{\lambda_{n_k}\}_{k\in\N}$
and an eigenfunction $g$ of the operator $L_{\A,a}$ associated to the
eigenvalue  $\gamma$
such that, for all $\tau\in (0,1)$,
\begin{equation*}
\lim_{k\to+\infty}\int_\tau^1
\bigg\|\frac{\tilde u(\lambda_{n_k}x,\lambda_{n_k}^2t)}{\sqrt{H(\lambda_{n_k}^2)}}
-t^{\gamma}g(x/\sqrt t)\bigg\|_{{\mathcal H}_t}^2dt=0
\end{equation*}
and
$$
\lim_{k\to+\infty}\sup_{t\in[\tau,1]}
\bigg\|\frac{\tilde u(\lambda_{n_k}x,\lambda_{n_k}^2t)}{\sqrt{H(\lambda_{n_k}^2)}}
-t^{\gamma}g(x/\sqrt t)\bigg\|_{{\mathcal L}_t}=0.
$$
\end{itemize}
\end{Lemma}
\begin{pf}
Let
\begin{align}\label{eq:33}
w_{\lambda}(x,t):=\frac{\tilde u_\lambda(x,t)}{\sqrt{H(\lambda^2)}},
\end{align}
with $\lambda\in(0,\sqrt T_0)$, so that $1<T_0/\lambda^2$.  From Lemma
\ref{l:Hcreas} it follows that, for all $t\in(0,1)$,
\begin{equation}\label{eq:21}
  \int_{\R^N}|w_{\lambda}(x,t)|^2G(x,t)\,dx
  =\frac{H(\lambda^2t)}{H(\lambda^2)}\leq t^{2C_1-\frac{N-2}{2}},
\end{equation}
with $C_1$  as in Corollary \ref{c:pos_per}.
Lemma \ref{l:Nabove}, Corollary \ref{c:pos_per}, and \eqref{eq:19} imply that
\begin{multline*}
  \frac1t \bigg(-\frac{N-2}4+{C_3}\bigg({\mathcal
    N}(T_0)+\frac{N-2}4\bigg)\bigg) H_\lambda(t)\geq
  \lambda^2D(\lambda^2 t)\\
  \geq \frac1t\bigg(C_1-\frac{N-2}{4}\bigg)H_\lambda(t)+C_1
  \int_{\R^N}\bigg(|\nabla \tilde u_{\lambda}(x,t)|^2
+\frac{|\tilde u_{\lambda}(x,t)|^2}{|x|^2}\bigg)G(x,t)\,dx
\end{multline*}
and hence, in view of \eqref{eq:21},
\begin{equation}\label{eq:20}
  t\int_{\R^N}\bigg(|\nabla w_{\lambda}(x,t)|^2
+\frac{|w_{\lambda}(x,t)|^2}{|x|^2}\bigg)
G(x,t)\,dx\leq
C_1^{-1} \Big({\textstyle{C_3\big({\mathcal
    N}(T_0)+\frac{N-2}4\big)-C_1}}\Big)t^{2C_1-\frac{N-2}{2}}
\end{equation}
for a.e. $t\in (0,1)$.
By scaling, we have that  the family of functions
\begin{equation*}
\widetilde{w}_{\lambda}(x,t)=w_{\lambda}(\sqrt{t}x,t)=\dfrac{\tilde u (\lambda \sqrt{t}x,
  \lambda^2t)}{\sqrt{H(\lambda^2)}}
\end{equation*}
satisfies
\begin{equation}\label{eq:22}
  \int_{\R^N} |\widetilde{w}_{\lambda}(x,t)|^{2} G(x,1)\,dx
  =\int_{\R^N} |w_{\lambda}(x,t) |^{2} G(x,t)\,dx
\end{equation}
and
\begin{equation}\label{eq:23}
  \int_{\R^N} \bigg(|\nabla\widetilde{w}_{\lambda}(x,t)|^{2}+\frac{|\widetilde{w}_{\lambda}(x,t)|^2}{|x|^2}\bigg)
G(x,1)\,dx
  =t\int_{\R^N}\bigg(|\nabla  w_{\lambda}(x,t)|^{2}+\frac{| w_{\lambda}(x,t)|^2}{|x|^2}\bigg)
G(x,t)\,dx.
\end{equation}
From \eqref{eq:21}, \eqref{eq:20}, \eqref{eq:22}, and \eqref{eq:23},
we deduce that, for all  $\tau \in (0,1)$,
\begin{equation}\label{eq:27}
  \big\{\widetilde{w}_{\lambda}\big\}_{\lambda\in(0,\sqrt T_0)}
  \text{ is bounded in }
  L^\infty(\tau,1;{\mathcal H})
\end{equation}
uniformly with respect to $\lambda\in(0,\sqrt T_0)$.
Moreover $\widetilde{w}_{\lambda}$ solves the following weak equation: for all $\phi\in{\mathcal H}$,
\begin{multline}\label{eq:25}
  {\phantom{\big\langle}}_{{\mathcal H}^\star}\big\langle
  (\widetilde{w}_{\lambda})_t,\phi
  \big\rangle_{{\mathcal H}}
  =\frac1t\int_{\R^N}\bigg(\nabla_\A \widetilde{w}_{\lambda}(x,t) \cdot
  \overline{\nabla_\A \phi(x)}-
  \dfrac{a\big(\frac{x}{|x|}\big)}{|x|^2}\,\widetilde{w}_{\lambda}(x,t)
  \overline{\phi(x)}\\
  -\lambda^2t\, h\Big(\lambda\sqrt tx, -\lambda^2 t)\widetilde{w}_{\lambda}(x,t)\overline{\phi(x)}\bigg)G(x,1)\,dx.
\end{multline}
From (\ref{eq:h}) it follows that
\begin{align}\label{eq:26}
  &\lambda^2\left|\int_{\R^N}h(\lambda\sqrt
    tx, -\lambda^2 t)\widetilde{w}_{\lambda}(x,t)\overline{\phi(x)}G(x,1)\,dx\right|\\
  \notag&\leq
  C_h\lambda^2\int_{\R^N}|\widetilde{w}_{\lambda}(x,t)||\phi(x)|G(x,1)\,dx
  +C_h\frac{\lambda^\e}{t}\int_{\R^N}|x|^{-2+\e}
  |\widetilde{w}_{\lambda}(x,t)||\phi(x)|G(x,1)\,dx\\
  \notag& \leq
  C_h\lambda^2\|\widetilde{w}_{\lambda}(\cdot,t)\|_{\mathcal H}
  \|\phi\|_{\mathcal H}+ C_h\frac{\lambda^\e}{t}\int_{|x|\leq1} \frac{
    |\widetilde{w}_{\lambda}(x,t)||\phi(x)|}{|x|^2}G(x,1)\,dx\\
  \notag& \hskip6cm +C_h\frac{\lambda^\e}{t}\int_{|x|\geq1}
  |\widetilde{w}_{\lambda}(x,t)||\phi(x)|G(x,1)\,dx\\
  \notag&\leq
  C_h\frac{\lambda^\e}{t}\big(t\,\lambda^{2-\e}+2  \big) \|\widetilde{w}_{\lambda}(\cdot,t)\|_{\mathcal H}
  \|\phi\|_{\mathcal H}
\end{align}
for all $\lambda\in (0,\sqrt T_0)$ and a.e. $t\in(0,1)$. From
(\ref{eq:25}) and  (\ref{eq:26}), and Lemma \ref{Hardytemp} we deduce
that 
\begin{equation}\label{eq:89}
\|(\widetilde{w}_{\lambda})_t(\cdot,t)\|_{{\mathcal H}^\star}\leq
\frac{{\rm const}}{t}\|\widetilde{w}_{\lambda}(\cdot,t)\|_{\mathcal H}
\end{equation}
for all $\lambda\in (0,\sqrt T_0)$ and a.e. $t\in(0,1)$
and for some $\mathop{\rm const}>0$ independent of $t$ and $\lambda$.

In view of (\ref{eq:27}),  estimate (\ref{eq:89}) yields, for all
 $\tau \in (0,1)$,
\begin{equation}\label{eq:28}
  \big\{(\widetilde{w}_{\lambda})_t\big\}_{\lambda\in(0,\sqrt T_0)}
  \text{ is bounded in }
  L^\infty(\tau,1;{\mathcal H}^\star)
\end{equation}
uniformly with respect to $\lambda\in(0,\sqrt T_0)$.  From
(\ref{eq:27}), (\ref{eq:28}), and \cite[Corollary 8]{S}, we deduce
that $\big\{\widetilde{w}_{\lambda}\big\}_{\lambda\in(0,\sqrt T_0)}$  is relatively
compact in $C^0([\tau,1],{\mathcal L})$ for all $\tau\in (0,1)$.
 Therefore, by a diagonal argument, from any given sequence
 $\lambda_n\to 0^+$ we can extract a subsequence $\lambda_{n_k}\to0^+$ such that
\begin{equation}\label{eq:29}
  \widetilde{w}_{\lambda_{n_k}}\to\widetilde{w} \quad\text{in} \quad
  C^0([\tau,1],{\mathcal L})
\end{equation}
for all $\tau\in(0,1)$ and for some $\widetilde{w}\in
\bigcap_{\tau\in(0,1)}C^0([\tau,1],{\mathcal L})$.  From the fact that 
$1=\|\widetilde{w}_{\lambda_{n_k}}(\cdot,1)\|_{\mathcal L}$ and 
 the convergence (\ref{eq:29}) it follows that 
\begin{equation}\label{eq:41}
\|\widetilde{w}(\cdot,1)\|_{\mathcal L}=1.
\end{equation}
In particular $\widetilde{w}$ is nontrivial. Furthermore, by
(\ref{eq:27}) and (\ref{eq:28}), the subsequence can be chosen in such a way
that also
\begin{equation}\label{eq:30}
  \widetilde{w}_{\lambda_{n_k}}\weakly\widetilde{w} \quad\text{weakly in }
  L^2(\tau,1;{\mathcal H})
  \quad\text{and}\quad
  (\widetilde{w}_{\lambda_{n_k}})_t\weakly\widetilde{w}_t
  \quad\text{weakly in }
  L^2(\tau,1;{\mathcal H}^\star)
\end{equation}
for all $\tau\in(0,1)$, so that $\widetilde{w}\in
\bigcap_{\tau\in(0,1)}L^2(\tau,1;{\mathcal H})$ and
$\widetilde{w}_t\in \bigcap_{\tau\in(0,1)}L^2(\tau,1;{\mathcal
  H}^\star)$. Actually, we can prove that 
\begin{equation}\label{eq:31}
  \widetilde{w}_{\lambda_{n_k}}\to\widetilde{w} \quad\text{strongly in} \quad
  L^2(\tau,1;{\mathcal H})\quad\text{for all }\tau\in(0,1).
\end{equation}
To prove  claim \eqref{eq:31}, we notice that  (\ref{eq:30}) allows
passing to the limit in (\ref{eq:25}). Since 
(\ref{eq:26}) and (\ref{eq:27}) imply that the perturbation term
vanishes in such a limit, we conclude that 
\begin{equation}\label{eq:32}
  {\phantom{\big\langle}}_{{\mathcal
      H}^\star}\big\langle \widetilde{w}_t,\phi
  \big\rangle_{{\mathcal H}}=\frac1t\int_{\R^N}\bigg(\nabla_\A \widetilde{w}(x,t)
  \cdot \overline{\nabla_\A \phi(x)}-
  \dfrac{a\big(\frac{x}{|x|}\big)}{|x|^2}\,\widetilde{w}(x,t)
  \overline{\phi(x)}\bigg)G(x,1)\,dx
\end{equation}
for all $\phi\in{\mathcal H}$ and a.e. $t \in (0,1)$, i.e. $\widetilde
w$ is a weak solution to
\begin{equation*}
  \widetilde{w}_t+\frac1t\bigg(-\mathcal L_{\A,a}\widetilde{w}-
  \frac x2\cdot \nabla \widetilde{w}\bigg)=0.
\end{equation*}
Testing the difference between (\ref{eq:25}) and (\ref{eq:32}) with
$(\widetilde{w}_{\lambda_{n_k}}- \widetilde{w})$ and integrating
with respect to $t$ between $\tau$ and $1$, we obtain
\begin{multline*}
  \int_\tau^1\bigg(\int_{\R^N}\bigg(|\nabla_\A (\widetilde{w}_{\lambda_{n_k}}-
  \widetilde{w})(x,t)|^2-\frac{a(x/|x|)}{|x|^2}\,
  |(\widetilde{w}_{\lambda_{n_k}}- \widetilde{w})(x,t)|^2\bigg)
  \,G(x,1)\,dx\bigg)dt\\
  = \frac{1}2\|\widetilde{w}_{\lambda_{n_k}}(1)-
  \widetilde{w}(1)\|^2_{\mathcal L}-
  \frac{\tau}2\|\widetilde{w}_{\lambda_{n_k}}(\tau)-
  \widetilde{w}(\tau)\|^2_{\mathcal L}-\frac12\int_\tau^1\bigg(\int_{\R^N}
  |(\widetilde{w}_{\lambda_{n_k}}-
  \widetilde{w})(x,t)|^2  \,G(x,1)\,dx\bigg)dt\\
  +\lambda_{n_k}^2\int_\tau^1t\bigg(\int_{\R^N}
  h(\lambda_{n_k}\sqrt t x, -\lambda_{n_k}^2
  t)\Re\big(\widetilde{w}_{\lambda_{n_k}}(x,t)
\overline{(\widetilde{w}_{\lambda_{n_k}}- \widetilde{w})(x,t)}\big) G(x,1)\,dx \bigg)dt.
\end{multline*}
Then (\ref{eq:26}) and (\ref{eq:29}) imply that,
for all $\tau\in (0,1)$,
\begin{equation*}
\lim_{k\to+\infty} 
\int_\tau^1\bigg(\int_{\R^N}\bigg(|\nabla_\A (\widetilde{w}_{\lambda_{n_k}}-
  \widetilde{w})(x,t)|^2-\frac{a(x/|x|)}{|x|^2}\,
  |(\widetilde{w}_{\lambda_{n_k}}- \widetilde{w})(x,t)|^2\bigg)
  \,G(x,1)\,dx\bigg)dt=0,
\end{equation*}
which yields the convergence claimed in (\ref{eq:31}) in view of
Corollary \ref{c:pos_def} and (\ref{eq:29}).
Thus, we have obtained that, for all $\tau\in(0,1)$,
\begin{equation}\label{eq:35}
\lim_{k\to+\infty}\int_\tau^1
\|{w}_{\lambda_{n_k}}(\cdot,t)-w(\cdot,t)\|_{{\mathcal
    H}_t}^2dt=0\quad\text{and}\quad 
\lim_{k\to+\infty}\sup_{t\in[\tau,1]}
\|{w}_{\lambda_{n_k}}(\cdot,t)-w(\cdot,t)\|_{{\mathcal L}_t}=0,
\end{equation}
where
$w(x,t):=\widetilde{w}\big(\frac{x}{\sqrt t },t\big)$ 
is a weak solution to
\begin{equation}\label{eq:limit_equation}
w_t-\mathcal L_{\A,a}w=0.
\end{equation}
In view of (\ref{eq:34}) and (\ref{eq:33}), we can write $\mathcal
N_\lambda$ as 
\begin{equation*}
  {\mathcal N}_\lambda(t)
  =\frac{ t\int_{\R^N}\Big(|\nabla_\A w_{\lambda}(x,t)|^2-
    \frac{a(x/|x|)}{|x|^2}|w_{\lambda}(x,t)|^2-\lambda^2
    h\big(\lambda x,-\lambda^2 t)|w_{\lambda}(x,t)|^2\Big)G(x, t)\,dx}{\int_{\R^N}|w_{\lambda}(x,t)|^2
    G(x,t)\,dx}
\end{equation*}
for all $t\in(0,1)$.  By \eqref{eq:35}
${w}_{\lambda_{n_k}}(\cdot,t)\to w(\cdot,t)$ in ${\mathcal H}_t$ for
a.e. $t\in (0,1)$, and, by \eqref{eq:26},
\begin{equation*}
t\lambda^2_{n_k}\int_{\R^N}
    h(\lambda_{n_k} x,-\lambda_{n_k}^2 t)|w_{\lambda_{n_k}}(x,t)|^2G(x, t)\,dx
\to 0
\end{equation*}
for a.e. $t\in (0,1)$. Hence 
\begin{multline}\label{eq:43}
  \int_{\R^N}\Big(|\nabla_\A w_{\lambda_{n_k}}(x,t)|^2-
    \frac{a(x/{|x|})}{|x|^2}|w_{\lambda_{n_k}}(x,t)|^2-\lambda_{n_k}^2
    h\big(\lambda_{n_k} x,-\lambda_{n_k}^2 t)|w_{\lambda_{n_k}}(x,t)|^2\Big)G(x, t)\,dx\\ \to D_w(t)
\end{multline}
and
\begin{equation}\label{eq:44}
\int_{\R^N}|w_{\lambda_{n_k}}(x,t)|^2
  G(x,t)\,dx\to H_w(t)
\end{equation}
as $k\to+\infty$, for a.e. $t\in (0,1)$, where
\[
D_w(t)=\int_{\R^N}\bigg(|\nabla_\A w(x,t)|^2-
  \frac{a(x/|x|)}{|x|^2}|w(x,t)|^2\bigg)G(x, t)\,dx
\ \text{ and }\
H_w(t)=\int_{\R^N}|w(x,t)|^2
  G(x,t)\,dx.
\]
From (\ref{eq:41}) it follows  that 
  $\int_{\R^N}|w(x,1)|^2G(x,1)\,dx=1$,
 which,
arguing as in Proposition \ref{p:t_0=0},  implies
that $H_w(t)>0$ for all $t\in(0,1)$.
From (\ref{eq:43}) and (\ref{eq:44}), it follows that
\begin{equation}\label{eq:36}
{\mathcal N}_{\lambda_{n_k}}(t)\to {\mathcal N}_w(t)
\quad\text{for a.e. }t\in(0,1),
\end{equation}
where ${\mathcal N}_w$ is the Almgren frequency function associated to 
equation \eqref{eq:limit_equation}, i.e.
\begin{equation}\label{eq:Nw}
{\mathcal N}_w(t)=\frac{ tD_w(t)}
{H_w(t)},
\end{equation}
which is well defined on $(0,1)$ since $H_w(t)>0$ for all $t\in(0,1)$
as observed above.

On the other hand, \eqref{eq:scale_for_N} implies that ${\mathcal
  N}_{\lambda_{n_k}}(t)={\mathcal N}(\lambda^2_{n_k}t)$ for all $t\in
(0,1)$ and $k\in\N$.  Fixing $t\in (0,1)$ and passing to the limit as
$k\to+\infty$, from Lemma \ref{l:limit} we obtain
\begin{equation}\label{eq:37}
{\mathcal N}_{\lambda_{n_k}}(t)\to \gamma
\quad\text{for all }t\in(0,1).
\end{equation}
Combining \eqref{eq:36} and \eqref{eq:37}, we deduce that
\begin{equation}\label{eq:Nw-gamma}
{\mathcal N}_w(t)=\gamma\quad \text{for all }t\in(0,1).
\end{equation}
Therefore ${\mathcal N}_w$ is constant in $(0,1)$ and hence ${\mathcal
  N}_w'(t)=0$ for any $t\in (0,1)$.  By \eqref{eq:limit_equation} and
Lemma~\ref{l:Nprime} with $h\equiv 0$, we obtain that 
\[
\left(w_t(\cdot,t)+\frac{\nabla w(\cdot,t)\cdot
    x}{2t},w(\cdot,t)\right)^2_{{\mathcal L}_t}
=\left\|w_t(\cdot,t)+\frac{\nabla w(\cdot,t)\cdot
    x}{2t}\right\|^2_{{\mathcal L}_t} \|w(\cdot,t)\|^2_{{\mathcal
    L}_t},
\]
where $(\cdot,\cdot)_{{\mathcal L}_t}$ denotes the scalar product in
${\mathcal L}_t$.  Then there exists a  function $\beta:(0,1)\to\R$ such that
\begin{equation}\label{eq:38}
w_t(x,t)+\frac{\nabla w(x,t)\cdot x}{2t} =\beta(t)w(x,t)\quad\text{for
  a.e. }t\in(0,1)\text{ and a.e. }x\in\R^N.
\end{equation}
Testing \eqref{eq:limit_equation} with $\phi=w(\cdot,t)$ in the sense
of \eqref{eq:defsol2} and taking into account \eqref{eq:38},
we obtain that
\begin{align*}
D_w(t)=
{\phantom{\bigg\langle}}_{{\mathcal
    H}_t^\star}\bigg\langle
w_t(\cdot,t)+\frac{\nabla w(\cdot,t)\cdot x}{2t},w(\cdot,t)
\bigg\rangle_{{\mathcal H}_t}=\beta(t)H_w(t),
\end{align*}
which, in view (\ref{eq:Nw}) and (\ref{eq:Nw-gamma}), yields
$\beta(t)=\frac{\gamma}{t}$ for a.e. $t\in(0,1)$. Then (\ref{eq:38}) becomes
\begin{equation}\label{eq:45}
w_t(x,t)+\frac{\nabla w(x,t)\cdot x}{2t}
=\frac{\gamma}{t}\,w(x,t)\quad \text{for a.e. }(x,t)\in\R^N\times
(0,1)
\end{equation}
in a distributional sense.
 From (\ref{eq:45}) and (\ref{eq:limit_equation}) we conclude that
\begin{equation}\label{eq:46}
\mathcal L_{\A,a} w
+\frac{\nabla w(x,t)\cdot x}{2t}=
\frac{\gamma}{t}\,w(x,t)
\end{equation}
for a.e. $(x,t)\in\R^N\times
(0,1)$  and in a weak  sense.
From (\ref{eq:45}), it follows that, letting, for all $\eta>0$ and a.e.
$(x,t)\in\R^N\times(0,1)$, $w^\eta(x,t):=w(\eta x,\eta^2t)$,
there holds
$\frac{d
  w^\eta}{d\eta}=\frac{2\gamma}{\eta}w^\eta
$
a.e. and in a
distributional sense. An integration yields
\begin{equation}
  \label{eq:weta}
 w^\eta(x,t)=w(\eta x,\eta^2t)=\eta^{2\gamma}w(x,t) \quad\text{for all }\eta>0
\text{ and a.e. }(x,t)\in\R^N\times(0,1).
\end{equation}
The function $g(x)=w(x,1)$ satisfies $g\in{\mathcal L}$, $\|g\|_{{\mathcal L}}=1$,
 and,
from (\ref{eq:weta}),
\begin{equation}\label{eq:47}
w(x,t)=w^{\sqrt t }\Big(\frac{x}{\sqrt t
},1\Big)=t^{\gamma}w\Big(\frac{x}{\sqrt t },1\Big)=
t^{\gamma}g\Big(\frac{x}{\sqrt t}\Big)
\quad\text{for a.e. }(x,t)\in\R^N\times(0,1).
\end{equation}
In particular, from (\ref{eq:47}) it follows that $g\big({\cdot}/{\sqrt t}\big)\in
{\mathcal H}_t$ for a.e. $t\in (0,1)$ and hence, by scaling,
$g\in{\mathcal H}$. Moreover, from (\ref{eq:46}) and (\ref{eq:47}), we obtain
that $g\in{\mathcal H}\setminus\{0\}$ weakly solves
\begin{equation*}
  \mathcal L_{\A,a} g(x)
+\frac{\nabla g(x)\cdot x}{2}=
\gamma g(x),
\end{equation*}
i.e.
$\gamma$ is an eigenvalue of the operator $L_{\A,a}$ defined in (\ref{eq:13})
and $g$ is an eigenfunction of $L_{\A,a}$ associated to $\gamma$. The proof
is now complete.
\end{pf}

\noindent The next step is to determine the asymptotic behaviour of
the normalization term in the blow-up family \eqref{eq:33}. To this
aim, in the next two lemmas we study the limit $\lim_{t\to
  0^+}t^{-2\gamma}H(t)$.

\begin{Lemma} \label{l:limite} Under the same assumptions as in Lemma
  \ref{l:blow_up}, let $\gamma:=\lim_{t\rightarrow 0^+} {\mathcal N}(t)$ be as in
  Lemma \ref{l:limit}.  Then the limit
$\lim_{t\to 0^+}t^{-2\gamma}H(t)$ 
exists and it is finite.
\end{Lemma}
\begin{pf}
We omit the proof, since it is very similar to that of \cite[Lemma 5.2]{FP}.
\end{pf}

\noindent The limit  $\lim_{t\to
  0^+}t^{-2\gamma}H(t)$ is indeed strictly positive, as we prove below.

\begin{Lemma}\label{limite_pos}
  Under the same assumptions as in Lemma \ref{l:blow_up}, we have that
  $\lim_{t\to 0^+}t^{-2\gamma}H(t)>0$.
\end{Lemma}
\begin{pf}
  Let us argue by contradiction and assume that $\lim_{t\to
    0^+}t^{-2\gamma}H(t)=0$. Let us consider the orthonormal basis of ${\mathcal L}$
  introduced in Remark \ref{rem:ortho} and given by $\{\widetilde V_{n,j}:
  j,n\in\N,j\geq 1\}$.  Since
  $\tilde u_{\lambda}(x,1)=\tilde u(\lambda x,\lambda^2)\in{\mathcal L}$ for all
  $\lambda\in (0,\sqrt T_0)$, we can expand $\tilde u_{\lambda}$ as
\begin{equation}\label{eq:70}
  \tilde u_{\lambda}(x,1)=\sum_{\substack{m,k\in\N\\k\geq1}}
  u_{m,k}(\lambda)
  \widetilde V_{m,k}(x)\quad \text{in }{\mathcal L},
\end{equation}
where
\begin{equation}\label{eq:68}
  u_{m,k}(\lambda)=\int_{\R^N}
  \tilde u_{\lambda}(x,1)\overline{\widetilde V_{m,k}(x)}G(x,1)\,dx.
\end{equation}
From \eqref{probtilde} and the fact that $\widetilde V_{m,k}(x)$ is an
eigenfuntion of the operator $L_{\A,a}$ associated to the eigenvalue
$\gamma_{m,k}$ defined in \eqref{eq:65}, we obtain that 
\begin{equation}\label{eq:42}
u'_{m,k}(\lambda)=\frac{2}{\lambda}\gamma_{m,k}u_{m,k}(\lambda)-2\lambda
\xi_{m,k}(\lambda) \quad\text{for all }m,k\in\N,\ k\geq 1,
\end{equation}
a.e. and distributionally in $(0,\sqrt{T_0})$, where 
\begin{equation}\label{eq:69}
  \xi_{m,k}(\lambda)
  =\int_{\R^N}h(\lambda x,-\lambda^2)\tilde 
  u_{\lambda}(x,1)\overline{\widetilde V_{m,k}(x)}G(x,1)\,dx.
\end{equation}
From Lemma \ref{l:blow_up}, $\gamma$ is an eigenvalue of the operator
$L_{\A,a}$, hence, by Proposition \ref{p:explicit_spectrum}, there exist
$m_0,k_0\in\N$, $k_0\geq 1$, such that
$\gamma=\gamma_{m_0,k_0}=m_0-\frac{\alpha_{k_0}}2$.  Let us denote as
$E_0$ the associated eigenspace and by $J_0$ the finite set of indices
$\{(m,k)\in\N\times(\N\setminus\{0\}):\gamma=m-\frac{\alpha_{k}}2\}$,
so that $\#J_0$ is equal to the multiplicity of $\gamma$
and an orthonormal basis of $E_0$ is given by $\{\widetilde V_{m,k} :
(m,k)\in J_0\}$.

In view of \eqref{eq:h}, for all $(m,k)\in J_0$ we
  can estimate $\xi_{m,k}$ as
\begin{gather}\label{eq:54}
  |\xi_{m,k}(\lambda)|\leq
  C_h\int_{\R^N}(1+\lambda^{-2+\e}|x|^{-2+\e})|\tilde u(\lambda x,\lambda^2)|
  |\widetilde V_{m,k}(x)|G(x,1)\,dx\\
  \notag\leq C_h\bigg(\int_{\R^N}|\tilde u(\lambda x,\lambda^2)|^2
  G(x,1)\,dx\bigg)^{\!\!1/2}+ C_h\lambda^{-2+\frac\e2}\int_{|x|\leq
    \lambda^{-1/2}}\frac{|\tilde u(\lambda x,\lambda^2)| |\widetilde
    V_{m,k}(x)|}{|x|^2}G(x,1)\,dx\\
  \notag\qquad+C_h\lambda^{-1+\frac\e2}\int_{|x|\geq
          \lambda^{-1/2}}|\tilde u(\lambda
  x,\lambda^2)|
  |\widetilde V_{m,k}(x)|G(x,1)\,dx\\
  \notag\hskip-0.5cm\leq C_h(1+\lambda^{-1+\frac\e2})\sqrt{H(\lambda^2)}
  +C_h\lambda^{-2+\frac\e2}\bigg(\int_{\R^N}\!\!{\textstyle{\frac{|\tilde u(\lambda
    x,\lambda^2)|^2} {|x|^2}}}G(x,1)\,dx\bigg)^{\!\!\frac12}
  \bigg(\int_{\R^N}\!\!{\textstyle{\frac{|\widetilde V_{m,k}(x)|^2}
  {|x|^2}G(x,1)\,dx}}\bigg)^{\!\!\frac12}.
\end{gather}
Corollary \ref{c:pos_per} and Lemma \ref{l:Nabove} imply that
\begin{multline}\label{eq:55}
  \int_{\R^N}\frac{|\tilde u(\lambda x,\lambda^2)|^2} {|x|^2}G(x,1)\,dx
  =\lambda^2\int_{\R^N}\frac{|\tilde u(y,\lambda^2)|^2} {|y|^2}G(y,\lambda^2)\,dy
  \\\leq \frac{\lambda^2}{C_1}\bigg(D(\lambda^2)+
  \frac{C_2}{\lambda^2}H(\lambda^2)\bigg)=
  \frac{H(\lambda^2)}{C_1}\big({\mathcal N}(\lambda^2)+C_2\big)
=O(H(\lambda^2))
\end{multline}
as $\lambda\to0^+$. On the other hand, from Lemma \ref{Hardy_aniso}
it follows that, for all $(m,k)\in J_0$,
\begin{equation}\label{eq:56}
\int_{\R^N}\frac{|\widetilde V_{m,k}(x)|^2}
  {|x|^2}G(x,1)\,dx\leq
\bigg(\mu_1(\A,a)+\frac{(N-2)^2}4\bigg)^{\!\!-1}\bigg(\gamma+\frac{N-2}4
\bigg).
\end{equation}
Combining (\ref{eq:54}), (\ref{eq:55}), (\ref{eq:56}),  and Lemma \ref{stimaH},
we obtain that
\begin{align}\label{eq:57}
|\xi_{m,k}(\lambda)|\leq
  O(\lambda^{-2+\frac{\e}{2}+2\gamma})\quad\text{as }\lambda\to0^+,
\end{align}
for all $(m,k)\in
J_0$.

In order to prove the theorem, we argue by contradiction and assume that $\lim_{t\to
    0^+}t^{-2\gamma}H(t)=0$. By orthogonality of the $\widetilde V_{m,k}$'s in ${\mathcal L}$,
we have that
$$
H(\lambda^2)=\sum_{\substack{n,j\in\N\\j\geq1}} (u_{n,j}(\lambda) )^2
\geq (u_{m,k}(\lambda) )^2\quad\text{for all }
\lambda\in(0,\sqrt{T_0})\text{ and }m,k\in\N,\ k\geq1.
$$
Hence,  $\lim_{t\to
    0^+}t^{-2\gamma}H(t)=0$ implies that
\begin{equation}\label{eq:58}
\lim_{\lambda\to 0^+}\lambda^{-2\gamma}u_{m,k}(\lambda)=0
\quad\text{for all }
m,k\in\N,\ k\geq1.
\end{equation}
Integration of \eqref{eq:42} over $(\rho,\lambda)$ yields that, for all
$\lambda,\rho\in(0,\sqrt{T_0})$,
\begin{equation}\label{eq:59}
u_{m,k}(\rho)=
\rho^{2\gamma_{m,k}}\left(\lambda^{-2\gamma_{m,k}} u_{m,k}(\lambda)+
2\int_{\rho}^{\lambda}s^{1-2\gamma_{m,k}}\xi_{m,k}(s)\,ds\right).
\end{equation}

Estimate \eqref{eq:57} implies that, for  all $(m,k)\in
J_0$, the function $s\mapsto
s^{1-2\gamma}\xi_{m,k}(s)$ belongs to $L^1(0,\sqrt{T_0})$. Letting
$\rho\to 0^+$ in (\ref{eq:59}) and using (\ref{eq:58}), we obtain that
\begin{equation}\label{eq:60}
u_{m,k}(\lambda)=
-2\lambda^{2\gamma}\int_0^\lambda s^{1-2\gamma}\xi_{m,k}(s)\,ds,
\quad\text{for all
$\lambda\in(0,\sqrt{T_0})$.}
\end{equation}
From (\ref{eq:57}) and (\ref{eq:60}), we  deduce that, for all
$(m,k)\in J_0$,
\begin{equation}\label{eq:61}
u_{m,k}(\lambda)=O(\lambda^{2\gamma+\frac\e2})\quad\text{as }\lambda\to0^+.
\end{equation}
Let us fix $0<\sigma<\frac\e2$. By Lemma \ref{stimaH},
there exists $K_2(\sigma)$ such that
$H(\lambda^2)\geq K_2(\sigma)\lambda^{2(2\gamma+\sigma)}$ for  all 
$\lambda\in(0,\sqrt{T_0})$.
Therefore, (\ref{eq:61}) implies that, for all
$(m,k)\in J_0$,
\begin{equation}\label{eq:48}
\frac{u_{m,k}(\lambda)}{\sqrt{H(\lambda^2)}}=O(\lambda^{\frac\e2-\sigma})=o(1)
\quad\text{as }\lambda\to0^+.
\end{equation}
On the other hand, by Lemma \ref{l:blow_up}, for every sequence
$\lambda_n\to 0^+$, there exists a subsequence
$\{\lambda_{n_j}\}_{j\in\N}$ such
that
\begin{equation*}
\frac {\tilde u_{\lambda_{n_j}}(x,1)}{\sqrt{H(\lambda_{n_j}^2)}}
\to g\quad \text{in }{\mathcal L}\quad\text{as }j\to+\infty,
\end{equation*}
for some  eigenfunction $g\in
E_0\setminus\{0\}$.
Therefore, for all $(m,k)\in J_0$,
\begin{equation}\label{eq:63}
\frac{u_{m,k}(\lambda_{n_j})}{\sqrt{H(\lambda_{n_j}^2)}}
=\left(\frac {\tilde u_{\lambda_{n_j}}(x,1)}{\sqrt{H(\lambda_{n_j}^2)}},
\widetilde V_{m,k}\right)_{\mathcal L}\to(g,
\widetilde V_{m,k})_{\mathcal L}\quad\text{as }j\to+\infty.
\end{equation}
From (\ref{eq:48}) and (\ref{eq:63}) we deduce that
 $(g,\widetilde V_{m,k})_{\mathcal L}=0$ for all $(m,k)\in J_0$.
Since $g\in E_0$ and $\{\widetilde V_{m,k} : (m,k)\in J_0\}$
is  an orthonormal basis of $E_0$, this implies that $g=0$,
a contradiction.
\end{pf}

\begin{pfn}{Theorem \ref{asym1}}
As already observed, up to a translation it is not restrictive to
assume that $t_0=0$; then the analysis performed in this section
applies to the function $\tilde u(x,t)=u(x,-t)$. In particular
(\ref{eq:671}) follows from Lemma \ref{l:limit}, part (i) of Lemma
\ref{l:blow_up} and Proposition \ref{p:explicit_spectrum}, i.e. 
  there exists an eigenvalue
  $\gamma_{m_0,k_0}=m_0-\frac{\alpha_{k_0}}2$ of $L_{\A,a}$, $m_0,k_0\in\N$,
  $k_0\geq 1$,  such that $\gamma=\lim_{t\to
    0^+}{\mathcal N}(t)=\lim_{t\to
    0^-}\widetilde {\mathcal N}(t)=\gamma_{m_0,k_0}$.  Let $E_0$ be the associated
  eigenspace and $J_0$ the finite set of indices
  $\{(m,k)\in\N\times(\N\setminus\{0\}):\gamma_{m_0,k_0}=m-\frac{\alpha_{k}}2\}$,
  so that $\{\widetilde V_{m,k} : (m,k)\in J_0\}$   is an
  orthonormal basis of $E_0$. Let $\{\lambda_n\}_{n\in\N}\subset
  (0,+\infty)$ such that $\lim_{n\to+\infty}\lambda_n=0$. Then, from
  part (ii) of Lemma \ref{l:blow_up} and Lemmas \ref{l:limite} and
  \ref{limite_pos}, there exist a subsequence
  $\{\lambda_{n_k}\}_{k\in\N}$ and complex numbers
  $\{\beta_{n,j}:(n,j)\in J_0\}$ such that
  $\beta_{n,j}\neq 0$ for some $(n,j)\in J_0$ and, for any $\tau\in(0,1)$,
\begin{equation}\label{eq:74}
\lim_{k\to+\infty}\int_\tau^1
\bigg\|\lambda_{n_k}^{-2\gamma}\tilde u(\lambda_{n_k} x,\lambda_{n_k}^2t)
-t^{\gamma}\sum_{(n,j)\in J_0}\beta_{n,j}\widetilde V_{n,j}(x/\sqrt t)
\bigg\|_{{\mathcal H}_t}^2dt=0
\end{equation}
and
\begin{equation}\label{eq:75}
\lim_{k\to+\infty}\sup_{t\in[\tau,1]}
\bigg\|\lambda_{n_k}^{-2\gamma}\tilde u(\lambda_{n_k} x,\lambda_{n_k}^2t)
-t^{\gamma}\sum_{(n,j)\in J_0}\beta_{n,j}\widetilde V_{n,j}(x/\sqrt t)
\bigg\|_{{\mathcal L}_t}=0.
\end{equation}
In particular,
\begin{equation}\label{eq:71}
  \lambda_{n_k}^{-2\gamma}\tilde u(\lambda_{n_k} x,\lambda_{n_k}^2)
  \mathop{\longrightarrow}\limits_{k\to+\infty}
  \sum_{(n,j)\in J_0}\beta_{n,j}\widetilde V_{n,j}(x)
  \quad\text{in }{\mathcal L}.
\end{equation}
Let us fix $\Lambda\in(0,\sqrt{T_0})$ and define  $u_{m,j}$ and $\xi_{m,j}$
as in (\ref{eq:68}--\ref{eq:69}).  From
(\ref{eq:71}) and orthogonality of the $\widetilde V_{n,j}$'s
it follows that, for any $(m,j)\in J_0$, $\lim_{k\to+\infty}
\lambda_{n_k}^{-2\gamma}u_{m,j}(\lambda_{n_k})
=\beta_{m,j}$.
Therefore from (\ref{eq:59}) we have  that, for every $(m,j)\in J_0$,
\begin{align*}
  \beta_{m,j}&=\Lambda^{-2\gamma} u_{m,j}(\Lambda)+
  2\int_{0}^{\Lambda}s^{1-2\gamma}\xi_{m,j}(s)\,ds\\
  &=\Lambda^{-2\gamma}
  \int_{\R^N}\tilde u(\Lambda x,\Lambda^2)\overline{\widetilde V_{m,j}(x)}G(x,1)\,dx\\
  &\hskip3cm+2\int_0^{\Lambda}s^{1-2\gamma} \bigg( \int_{\R^N}h(s
  x,-s^2) \tilde u(sx,s^2)\overline{\widetilde V_{m,j}(x)}G(x,1)\,dx \bigg) ds.
\end{align*}
The above formula shows that the $\beta_{m,j}$'s depend neither on the sequence
$\{\lambda_n\}_{n\in\N}$ nor on its subsequence
$\{\lambda_{n_k}\}_{k\in\N}$. Then we can conclude that the convergences in
(\ref{eq:74}) and (\ref{eq:75}) actually hold as $\lambda\to 0^+$.
 The proof is thereby complete.
\end{pfn}

\begin{pfn}{Corollary \ref{cor:uniq_cont}}
 Let us argue by contradiction and assume that $u\not\equiv 0$ in $\R^N\times
  (t_0-T,t_0)$. Let $k\in\N$ be such that $k>\gamma_{m_0,k_0}$, being
  $\gamma_{m_0,k_0}$ as in Theorem \ref{asym1}. 
From (\ref{eq:uniq_cont}) it follows that,
  for a.e. $(x,t)\in\R^N\times (0,1)$,
\begin{equation*}
\lim_{\lambda\to 0^+}\lambda^{-2\gamma_{m_0,k_0}}t^{-\gamma_{m_0,k_0}}u(\lambda x,t_0-\lambda^2 t)=0.
\end{equation*}
On the other hand, Theorem \ref{asym1} implies that, 
for all $t\in(0,1)$ and a.e. $x\in\R^N$,
\[
\lim_{\lambda\to 0^+}\lambda^{-2\gamma_{m_0,k_0}}t^{-\gamma_{m_0,k_0}}u(\lambda x,t_0-\lambda^2 t)
= g(x/\sqrt t),
\]
for some  $g\in{\mathcal H}\setminus\{0\}$ eigenfunction of
$L_{\A,a}$ associated to the eigenvalue $\gamma_{m_0,k_0}$, thus
giving rise to a contradiction.
\end{pfn}

\section{Proof of Theorem \ref{t:representation}: representation
  formula of solutions}\label{sec:repr-form-solut}

Let $u$
be a weak solution to (\ref{prob}) with $h\equiv0$. By saying that $u$
is a weak solution to (\ref{prob}) we mean that the 
the function  $\varphi$ defined as 
\begin{equation}  \label{varphi}
\varphi(x,t)= u\big(\sqrt{1+t}\, x,t\big), \quad t\geq 0
\end{equation}
is a weak solution of the following equation (which is equivalent to
(\ref{prob}) up to the change of variable $x\mapsto \sqrt{1+t}\, x$):
\begin{equation}\label{varphieq}
  \dfrac{d \varphi}{d t}(x,t)= \dfrac{1}{(1+t)} \bigg({\mathcal{-L}}_{{
      \mathbf{A}},a}\varphi(x,t)+\frac{1}{2} \nabla \varphi(x,t)\cdot x\bigg),
\end{equation}
in the sense that 
\begin{equation}\label{eq:67}
\varphi\in L^2_{\rm loc}([0,+\infty),\tilde{\mathcal H} )\quad\text{and}\quad
\varphi_t\in L^2_{\rm loc}([0,+\infty),\tilde{\mathcal H}^\star ),
\end{equation}
and 
\begin{equation*}
{\phantom{\big\langle}}_{\tilde{\mathcal
    H}^\star}\big\langle
\varphi_t,w
\big\rangle_{\tilde{\mathcal H}}=-\frac1{t+1}
\int_{\R^N}\bigg(\nabla_\A \varphi(x,t)\cdot \overline{\nabla_\A w(x)}-
\dfrac{a(x/|x|)}{|x|^2}\,\varphi(x,t) \overline{w(x)}\bigg)e^{\frac{|x|^2}4}\,dx
\end{equation*}
for all $w\in \tilde{\mathcal H}$, where $\tilde{\mathcal H}$ is
defined in \eqref{eq:62}. From \eqref{eq:67} it follows
that $\varphi\in C^0([0,+\infty),\tilde{\mathcal L} )$.
Furthermore, \begin{equation*}
  \varphi (x, 0)= u(x,0)=u_0(x).
\end{equation*}
A representation formula for solutions $u$ to (\ref{prob}) with
$h\equiv 0$ can be found 
(in the spirit of \cite{FFFP})
by
expanding the transformed solution $\varphi$ to (\ref{varphieq}) in Fourier
series with respect to an orthonormal basis of $\tilde{\mathcal{L}}$
consisting of eigenfunctions of an Ornstein-Uhlenbeck type 
operator perturbed with singular homogeneous electromagnetic
potentials, i.e. of the operator $L^{-}_{{\mathbf{A}},a}$ defined in
\eqref{eq:Lmeno} and acting as 
\begin{equation}  \label{operator2}
{}_{\tilde{{\mathcal{H}}}^\star}\langle L_{{\mathbf{A}},a}^{-}\varphi,w \rangle_{{\tilde{\mathcal{H}}}}
\\
= \int_{{\mathbb{R}}^N}\bigg(\nabla_{{\mathbf{A}}} \varphi(x)\cdot\overline{%
\nabla_{{\mathbf{A}}} w(x)}-\frac{a(\frac{x}{|x|})}{|x|^2}\,\varphi(x)\overline{%
w(x)} \bigg)e^{\frac{|x|^2}{4}}\,dx,
\end{equation}
for all $\varphi,w\in{\tilde{\mathcal{H}}}$, 
where $\tilde{{\mathcal{H}}}^\star$ denotes the dual
space of $\tilde{{\mathcal{H}}}$ and ${}_{\tilde{{\mathcal{H}}}^\star}\langle
\cdot,\cdot\rangle_{\tilde{{\mathcal{H}}}}$ is the corresponding duality product.

Let us expand the initial datum $u_{0}=u(\cdot
,0)=\varphi (\cdot ,0)$  in Fourier series with
respect to the orthonormal basis of $\tilde {\mathcal{L}}$
introduced in Remark \ref{rem:tilde} as
\begin{equation}
u_{0}=\sum\limits_{\substack{ m,k\in {\mathbb{N}} \\ k\geq 1}}c_{m,k}%
\widetilde{U}_{m,k}\quad \text{in }\tilde{\mathcal{L}},\quad \text{where
}c_{m,k}=\int_{{\mathbb{R}}^{N}}u_{0}(x)\overline{\widetilde{U}_{m,k}(x)}%
e^{\frac{|x|^2}{4}}\,dx,  \label{datoinicial}
\end{equation}
and, for $t\geq 0$, the function $\varphi(\cdot ,t)$ defined
in \eqref{varphi} as
\begin{equation}
\varphi(\cdot ,t)=\sum\limits_{\substack{ m,k\in {\mathbb{N}} \\ k\geq 1}}%
\varphi _{m,k}(t)\widetilde{U}_{m,k}\quad \text{in }\tilde{\mathcal L},
\label{eq:exp_varphi}
\end{equation}%
where
\begin{equation*}
\varphi _{m,k}(t)=\int_{{\mathbb{R}}^{N}}\varphi (x,t)\overline{\widetilde{U}%
_{m,k}(x)}e^{\frac{|x|^2}{4}}\,dx.
\end{equation*}%
Since $\varphi(x,t)$ satisfies \eqref{varphieq}, by Remark
\ref{rem:tilde} we obtain that
$\varphi _{m,k}\in C^{1}([0,+\infty),{\mathbb{C}})$ and
\begin{equation*}
\varphi _{m,k}^{\prime }(t)=-\dfrac{\tilde\gamma _{m,k}}{1+t}\, \varphi
_{m,k}(t),\quad \varphi _{m,k}(0)=c_{m,k},
\end{equation*}%
with $\tilde \gamma_{m,k}=\frac N2+\gamma_{m,k}=
 \frac{N}{2}-\frac{\alpha_{k}}{2}+m$. Integration yields $\varphi
 _{m,k}(t)=c_{m,k}(1+t)^{-\tilde \gamma
_{m,k}}$. Hence expansion (\ref{eq:exp_varphi}) can be rewritten as
\begin{equation*}
\varphi (z,t)=\sum\limits_{\substack{ m,k\in {\mathbb{N}} \\ k\geq 1}}%
c_{m,k}(1+t)^{-\tilde \gamma _{m,k}}\widetilde{U}_{m,k}(z)\quad \text{in }%
\tilde{\mathcal L},\quad \text{for all }t\geq0.
\end{equation*}%
We notice that
$\tilde\gamma_{m,k}\geq0$ for all $m,k$, then, in view of the Parseval identity, for all $t\geq0$,
\begin{equation}\label{eq:64}
\|\varphi(\cdot,t)\|_{\tilde{\mathcal L}}^2=
\sum\limits_{\substack{ m,k\in {\mathbb{N}} \\ k\geq 1}}
c_{m,k}^2(1+t)^{-2\tilde \gamma _{m,k}}
\leq
\sum\limits_{\substack{ m,k\in {\mathbb{N}} \\ k\geq 1}}
c_{m,k}^2=\|u_0\|_{\tilde{\mathcal L}}^2.
\end{equation}
In view of \eqref{datoinicial}, the above series can be written as
\begin{equation*}
\varphi(z,t)=\sum\limits_{\substack{ m,k\in {\mathbb{N}} \\ k\geq 1}}%
(1+t)^{-\tilde\gamma _{m,k}}\bigg(\int_{{\mathbb{R}}^{N}}e^{\frac{|y|^2}{4}}u_{0}(y)\overline{%
\widetilde{U}_{m,k}(y)}\,dy\bigg)\widetilde{U}_{m,k}(z),
\end{equation*}%
in the sense that, for all $t\geq 0$, the above series converges
in $\tilde {\mathcal{L}}$. Since $u_{0}(y)$ can be expanded~as
\begin{equation*}
u_{0}(y)=u_{0}\big(|y|\,\tfrac{y}{|y|}\big)=\sum_{j=1}^{\infty
}u_{0,j}(|y|)\psi _{j}\big(\tfrac{y}{|y|}\big)\quad \text{in }L^{2}({\mathbb{%
S}}^{N-1}),
\end{equation*}%
where $u_{0,j}(|y|)=\int_{{\mathbb{S}}^{N-1}}u_{0}(|y|\theta )\overline{\psi
_{j}(\theta )}\,dS(\theta )$, we conclude that
\begin{align*}
\varphi (z,t)& =\sum\limits_{\substack{ m,k\in {\mathbb{N}} \\ k\geq 1}}%
\frac{(1+t)^{-\tilde\gamma _{m,k}}}{\Vert U_{m,k}\Vert
  _{\tilde{\mathcal L}}^{2}}%
U_{m,k}(z)\bigg(\int_{0}^{\infty }\!\!u_{0,k}(r)r^{N-1-\alpha _{k}}P_{k,m}(%
\tfrac{r^{2}}{4})\,dr\bigg) \\
& =\sum\limits_{k=1}^{\infty }\psi _{k}\big(\tfrac{z}{|z|}\big)\frac{
(1+t)^{\frac{\alpha_k}{2}-\frac{N}{2}}}{2^{1+2\beta _{k}}\Gamma (1+\beta _{k})}\!\left[
\sum\limits_{m=0}^{\infty }\frac{{\textstyle{\binom{m+\beta _{k}}{m}}}}{%
(1+t)^m}\times \right.  \\
& \ \ \ \left. \times \bigg(\int_{0}^{\infty }\!\frac{u_{0,k}(r)}{%
|rz|^{\alpha _{k}}}P_{k,m}\big(\tfrac{r^{2}}{4}\big)P_{k,m}\big(\tfrac{%
|z|^{2}}{4}\big)e^{-\frac{|z|^{2}}{4}}r^{N-1}dr\!\bigg)\!\right] \!.
\end{align*}%
By \cite{AS} we know that
\begin{equation*}
P_{k,m}\Big(\frac{r^{2}}{4}\Big)=\dfrac{\Gamma (1+\beta _{k})}{\Gamma
(1+\beta _{k}+m)}e^{\frac{r^{2}}{4}}r^{-\beta _{k}}2^{\beta _{k}}%
\displaystyle\int_{0}^{\infty }e^{-t}t^{m+\frac{\beta _{k}}{2}}J_{\beta
_{k}}(r\sqrt{t})\,dt,
\end{equation*}%
where $J_{\beta _{k}}$ is the Bessel function of the first kind of order $%
\beta _{k}$. Therefore,
\begin{align*}
\varphi (z,t)& =2\sum\limits_{k=1}^{\infty }\psi _{k}\big(\tfrac{z}{|z|}\big)%
(1+t)^{\frac{\alpha_k}{2}-\frac{N}{2} }  \Bigg[
\sum\limits_{m=0}^{\infty }\frac{(1+t)^{-m}}{
\Gamma (1+\beta _{k}+m)\Gamma (1+m)}\times  \\
& \ \ \ \times \!\bigg(\!\int_{0}^{\infty }\frac{u_{0,k}(r)}{|rz|^{\alpha
_{k}+\beta _{k}}}e^{\frac{r^{2}}{4}}\left( \int_{0}^{\infty
}\!\!\!\int_{0}^{\infty }\!\!e^{-s^{2}-{s^{\prime }}^{2}}\!(ss^{\prime})^{
2m+\beta _{k}+1}\times \right.  \\
& \ \ \ \left. \times J_{\beta _{k}}(rs)J_{\beta _{k}}(%
|z|s^{\prime })\,ds\,ds^{\prime }\right) \!r^{N-1}dr\!\bigg)\Bigg] \\
& =2\sum\limits_{k=1}^{\infty }\psi _{k}\big(\tfrac{z}{|z|}\big)
(1+t)^{\frac{\alpha_k}{2}-\frac{N}{2}}\!\Bigg[\int_{0}^{\infty }u_{0,k}(r)|rz|^{-\alpha
_{k}-\beta _{k}}e^{\frac{r^{2}}{4}}r^{N-1}\times  \\
& \ \ \ \times \bigg(%
\int_{0}^{\infty }\!\!\!\int_{0}^{\infty }\!\!\!\frac{ss^{\prime }}{%
e^{s^{2}+s^{\prime 2}}}\times  \\
& \ \ \ \times \bigg(\sum\limits_{m=0}^{\infty }\frac{(1+t)^{-m}}{\Gamma (1+m)\Gamma (1+\beta _{k}+m)}%
(ss^{\prime})^{ 2m+\beta _{k}}\bigg)\times  
  \ J_{\beta _{k}}(rs)J_{\beta _{k}}(|z|s^{\prime
})ds\,ds^{\prime }\!\bigg)dr\Bigg]\\
& =2\sum\limits_{k=1}^{\infty }i^{-\beta_k}\psi _{k}\big(\tfrac{z}{|z|}\big)
(1+t)^{\frac{\alpha_k}{2}-\frac{N}{2}+\frac{\beta_k}{2}}\!\Bigg[\int_{0}^{\infty }u_{0,k}(r)|rz|^{-\alpha
_{k}-\beta _{k}}e^{\frac{r^{2}}{4}}r^{N-1}\times  \\
& \ \ \ \times \bigg(%
\int_{0}^{\infty }\!\!\!\int_{0}^{\infty }\!\!\!\frac{ss^{\prime }}{%
e^{s^{2}+s^{\prime 2}}} J_{\beta _{k}}\Big(\frac{2iss'}{\sqrt{1+t}}\Big)
  \ J_{\beta _{k}}(rs)J_{\beta _{k}}(|z|s^{\prime
})ds\,ds^{\prime }\!\bigg)dr\Bigg],
\end{align*}%
where in the last line we have used that $$ (1+t)^{-m} (ss')^{2m+\beta_k}= (1+t)^{\frac{\beta_k}{2}} i^{-\beta_k} (-1)^m \Big( \dfrac{iss'}{\sqrt{1+t}}\Big)^{2m+\beta_k}, $$
and  $$\sum\limits_{m=0}^{\infty } \frac{(-1)^m}{\Gamma (1+m)\Gamma (1+\beta
  _{k}+m)}\Big(\frac{iss^{\prime}}{\sqrt {1+t}}\Big)^{2m+\beta _{k}}=J_{\beta _{k}}\Big(\frac{2iss^{\prime }}{\sqrt{1+t}}\Big).$$  
Then we have that 
\begin{equation}
\varphi (z,t)  \label{eq:2} 
=2\sum\limits_{k=1}^{\infty }   \dfrac{ e^{-i\frac{\pi}{2} \beta_k}\psi _{k}\big(\tfrac{z}{|z|}\big) }{(1+t)^{\frac{N}{4}+\frac{1}{2}}}
\!\Bigg[\int_{0}^{\infty }u_{0,k}(r)|rz|^{-\frac{%
N-2}{2}}e^{\frac{r^{2}}{4}}{\mathcal{I}}_{k,t}(r,|z|)r^{N-1}dr\Bigg],
\end{equation}%
where
\begin{equation*}
{\mathcal{I}}_{k,t}(r,|z|)=\int_{0}^{\infty }\!\!\!\int_{0}^{\infty
}\!\!ss^{\prime} e^{-s^{2}-s^{\prime 2}}J_{\beta _{k}}\Big(\frac{2iss^{\prime }}{\sqrt{1+t}}\Big) J_{\beta _{k}}(rs)J_{\beta _{k}}(|z|s^{\prime })\,ds\,ds^{\prime }.
\end{equation*}%
From \cite[formula (1), p. 395]{watson} (with $t=s^{\prime }$, $p=1$, $a=%
\frac{2is}{\sqrt{1+t}}$, $b=|z|$, $\nu =\beta _{k}$
which satisfy $\Re (\nu )>-1$ and $|\mathop{\rm arg}p|<\frac{\pi }{4}$), we
know that
\begin{equation*}
  \int_{0}^{\infty }\!\!s^{\prime} e^{-s^{\prime 2}} J_{\beta _{k}}\Big(\frac{2iss^{\prime }}{\sqrt{1+t}}\Big)  J_{\beta _{k}}(
  |z|s^{\prime })\,ds^{\prime } =\frac{1}{2}e^{-\frac{1}{4} (|z|^2-\frac{4s^2}{1+t})} I_{\beta
    _{k}}\bigg(\frac{i|z|s}{\sqrt{1+t}}\bigg),
\end{equation*}
where $I_{\beta _{k}}$ denotes the modified Bessel function of order $\beta
_{k}$. Hence
\begin{align*}
 {\mathcal{I}}_{k,t}(r,|z|) 
&  =\dfrac{1}{2}\displaystyle\int_{0}^{\infty }se^{-s^{2}}J_{\beta _{k}}(%
rs)e^{-\frac{|z|^{2}}{4} }e^{\frac{s^2}{1+t}} I_{\beta
_{k}}\Big(\frac{i|z|s}{\sqrt{1+t}}\Big)\,ds \\
& =\dfrac{1}{2} e^{-i \beta_k \frac{\pi}{2}} e^{\frac{-|z|^2}{4}}\displaystyle\int_{0}^{\infty }se^{-s^{2}}J_{\beta _{k}}(%
rs)e^{\frac{s^2}{1+t}} J_{\beta
_{k}}\Big(\frac{-|z|s}{\sqrt{1+t}}\Big)\,ds ,
\end{align*}%
since $I_{\nu }(x)=e^{-\frac{1}{2}\nu \pi i}J_{\nu }(xe^{\frac{\pi }{2}i})$
(see e.g. \cite[9.6.3, p. 375]{AS}). We obtain
\begin{equation*}
{\mathcal{I}}_{k,t}(r,|z|)=\dfrac{1}{2} e^{-i \beta_k \frac{\pi}{2}} e^{\frac{-|z|^2}{4}}\displaystyle\int_{0}^{\infty }s e^{-s^2\frac{t}{1+t} }  J_{\beta _{k}}(%
rs)J_{\beta
_{k}}\Big(\frac{-|z|s}{\sqrt{1+t}}\Big)\,ds.
\end{equation*}%
Applying \cite[formula (1), p. 395]{watson} (with $t=s$, $p^{2}=\frac {t}{t+1}$, $a=r$, $b= -\frac{|z|}{\sqrt{1+t}}$, $\nu =\beta
_{k}$ which satisfy $\Re (\nu )>-1$ and $|\mathop{\rm arg}p|<\frac{\pi }{4}$%
) and \cite[9.6.3, p. 375]{AS}, we obtain
\begin{align} \label{eq:calI}
 {\mathcal{I}}_{k,t}(r,|z|) & 
=\dfrac{1}{4}e^{-\frac{\beta _{k}}{2}\pi i} e^{-\frac{|z|^2}{4}} 
  e^{-\frac{r^2 (1+t) +|z|^2}{4t}} \frac{(t+1)}{t} I_{\beta _{k}}\bigg(\dfrac{-r|z| \sqrt{1+t}%
}{2t}\bigg)  \notag \\
&= \dfrac{1}{4}e^{i\beta _{k}\pi } e^{-\frac{|z|^2}{4}} 
  e^{-\frac{r^2 (1+t) +|z|^2}{4t}} \frac{(t+1)}{t} J_{\beta _{k}}\bigg(\dfrac{-ir|z| \sqrt{1+t}%
}{2t}\bigg) .  \notag
\end{align}%
From \eqref{eq:2} and the above identity we deduce
\begin{equation*}
\varphi (z,t) 
=\frac{1}{2}\sum\limits_{k=1}^{\infty
}\dfrac{\psi _{k}\big(\tfrac{z}{|z|}\big)e^{i \frac\pi2\beta _{k}}}{ t(1+t)^{\frac{N}{4}-\frac{1}{2}}}\!\Bigg[%
\int_{0}^{\infty }\frac{u_{0,k}(r)}{|rz|^{\frac{N-2}{2}}}e^ {-\frac{|z|^2}{4t}-\frac{|z|^2}{4} -\frac{r^2}{4t}}   J_{\beta _{k}}\bigg(  \dfrac{-ir|z| \sqrt{1+t}
}{2t}\bigg)r^{N-1}\,dr\Bigg] .
\end{equation*}%
Notice that, by replacing $\int_{0}^{\infty }$ with $\int_{0}^{R}$ we
obtain the series representation of the solution $\varphi_{R}(z,t)$
with initial datum $u_{0,R}(x)\equiv \chi _{R}(x)u_{0}(x)$, being $\chi _{R}(x)$
the characteristic function of the ball $B_R$ of radius $R$ centered at the
origin. 
Since $u_{0,R} \to u_0$ in $\tilde{\mathcal L}$ as $R\to+\infty$, 
from \eqref{eq:64} it follows that $\varphi(\cdot,t)=\lim_{R\rightarrow \infty }\varphi_{R}(\cdot,t)$ in
$\tilde{\mathcal L}$. Recalling the definition of $u_{0,k}$ 
and observing that the queue of the series
\begin{equation*}
\sum\limits_{k=1}^{\infty
}\psi _{k}\big(\tfrac{z}{|z|}\big)\overline{\psi _{k}\big(\tfrac{y}{|y|}\big)}
e^{i \frac\pi2\beta _{k}}\frac{u_{0}(y)}{(|y||z|)^{\frac{N-2}{2}}}e^
{-\frac{|y|^2+|z|^2(1+t)}{4t}} 
  J_{\beta _{k}}\bigg(  \dfrac{-i|y||z| \sqrt{1+t}
}{2t}\bigg)
\end{equation*}
is uniformly convergent on compact sets (see \cite[Lemma 1.2]{FFFP}), we can exchange integral and sum and
write
\begin{multline*}
\varphi_R (z,t)  
\\=\frac{1}{2
  t(1+t)^{\frac{N}{4}-\frac{1}{2}}}\int_{B_R}\frac{u_0(y)}{(|y||z|)^{\frac{N-2}{2}}}
e^ {-\frac{|z|^2(1+t)+|y|^2}{4t}}\bigg[
\sum\limits_{k=1}^{\infty
}e^{i \frac\pi2\beta _{k}}\psi _{k}\big(\tfrac{z}{|z|}\big) \overline{\psi _{k}\big(\tfrac{y}{|y|}\big)}J_{\beta _{k}}\bigg(  \dfrac{-i|y||z| \sqrt{1+t}
}{2t}\bigg)\bigg]\,dy,
\end{multline*}
i.e.
\[
\varphi_R (z,t)  
=t^{-\frac N2}\int_{B_R}u_0(y)K\left(\tfrac{y}{\sqrt t},\tfrac{z\sqrt{1+t}}{\sqrt
    t}\right)\,dy
\]
where $K(y,z)$ is the kernel defined in \eqref{eq:3}.
Letting $R\to+\infty$ we obtain that 
\[
\varphi (z,t)  
=t^{-\frac N2}\int_{\R^N}u_0(y)K\left(\tfrac{y}{\sqrt t},\tfrac{z\sqrt{1+t}}{\sqrt
    t}\right)\,dy,
\]
where  the integral at the right hand side is understood in the sense
of  improper multiple integrals.
From \eqref{varphi} we conclude that, for all $t> 0$,
\[
  u(x,t)=t^{-\frac{N}{2}}
  \int_{\R^N}u_{0}(y)  K\bigg(\frac{y}{\sqrt{t}},\frac{x}{\sqrt{t}}\bigg) \,dy .
  \]


\begin{thebibliography}{ABCD}



\bibitem{AS}
M. Abramowitz, I. A. Stegun, 
   \textit{Handbook of mathematical functions with formulas, graphs, and
   mathematical tables.} National Bureau of Standards Applied
 Mathematics Series {\bf 55}. For sale by the Superintendent of Documents,
 U.S. Government Printing Office, Washington, D.C. 1964.


\bibitem{AV} G. Alessandrini, S. Vessella, {\it Remark on the strong
   unique continuation property for parabolic operators},
  Proc. Amer. Math. Soc.,  132 (2004), no. 2, 499--501.

\bibitem{almgren} F. J. Jr. Almgren, {\it Dirichlet's problem for
   multiple valued functions and the regularity of mass minimizing
    integral currents}, in ``Minimal submanifolds and geodesics''
    (Proc. Japan-United States Sem., Tokyo, 1977), pp. 1--6,
    North-Holland, Amsterdam-New York, 1979.
   

\bibitem{cazacu} C. Cazacu, D. Krej\v ci\v r\'\i, {\it 
The Hardy inequality and the heat equation with magnetic field in any
dimension},
 Comm. Partial Differential Equations 41 (2016), no. 7, 1056--1088. 

\bibitem{CH} X. Y. Chen, {\it A strong unique continuation theorem for
    parabolic equations}, Math. Ann., 311 (1998), no. 4, 603--630.

\bibitem{Djrbashian} A. E. 
Djrbashian, {\it Integral representations of solutions of the heat equation},
Izv. Nats. Akad. Nauk Armenii Mat., 37 (2002), no. 5, 3--11;
translation in 
J. Contemp. Math. Anal., 37 (2002), no. 5, 14--22 (2003).


\bibitem{E} L. Escauriaza, {\it Carleman inequalities and the heat
   operator}, Duke Math. J.,  104 (2000), no. 1, 113--127.

 \bibitem {EF} L. Escauriaza, F.J. Fern\'andez, {\it Unique
   continuation for parabolic operators}, Ark. Mat.,  41 (2003),
 no. 1, 35--60.


\bibitem{EFV} L. Escauriaza, F. J.  Fern\'andez, S.  Vessella, {\it
    Doubling properties of caloric functions},  Appl. Anal.,  85 (2006),
    no. 1-3, 205--223.

\bibitem {EKPV} L. Escauriaza, C. E. Kenig, G. Ponce, L. Vega, {\it
   Decay at infinity of caloric functions within characteristic
   hyperplanes}, Math. Res. Lett.,  13 (2006), no. 2-3, 441--453.

\bibitem{EV} L. Escauriaza, L. Vega, {\it Carleman inequalities and
   the heat operator. II}, Indiana Univ. Math. J., 50 (2001), no. 3,
 1149--1169.

\bibitem{ES} M. Escobedo, O. Kavian, {\it Variational problems
    related to self-similar solutions of the heat equation},
  Nonlinear Anal., 11 (1987), no. 10, 1103--1133.


\bibitem{FFFP} L. Fanelli, V. Felli, M. A. Fontelos, A. Primo, {\it 
Time decay of scaling critical electromagnetic Schr\"odinger flows},
 Comm. Math. Phys., 324 (2013), no. 3, 1033--1067. 

\bibitem{FFT} V. Felli, A. Ferrero, S. Terracini, {\it Asymptotic
    behavior of solutions to Schr\"odinger equations near an isolated
    singularity of the electromagnetic potential}, Journal of the
    European Mathematical Society,  13 (2011), n. 1, 119--174.

\bibitem {FE} F.J. Fern\'andez, {\it Unique continuation for parabolic
   operators. II}, Comm. Partial Differential Equations 28 (2003),
 no. 9-10, 1597--1604.

\bibitem{FP}  V. Felli, A. Primo, {\it Classification of local asymptotics for solutions to heat equations with inverse-square potentials}, Discrete Contin. Dyn. Syst. 31 (2011), no. 1, 65--107. 

 
\bibitem{GL} N. Garofalo, F.-H.  Lin, {\it Monotonicity properties of
   variational integrals, $A\sb p$ weights and unique continuation},
 Indiana Univ. Math. J.,  35 (1986), no. 2, 245--268.


\bibitem{kovarik}
H. Kova\v r\'\i k, H., Hynek {\it Heat kernels of two-dimensional magnetic Schr\"odinger
and Pauli operators}, Calc. Var. Partial Differential Equations, 44
(2012), no. 3-4, 351--374.

\bibitem{kukavica} I. Kukavica, {\it Backward uniqueness for solutions of linear parabolic
 equations}, Proc. Amer. Math. Soc., 132 (2004), no. 6, 1755--1760.

\bibitem{lw}
A. Laptev, T. Weidl, {\it Hardy inequalities for
      magnetic Dirichlet forms}, Mathematical results in quantum
    mechanics (Prague, 1998), 299--305; \textit{Oper. Theory Adv. Appl.} {108},
    Birkh\"auser, Basel, 1999.

\bibitem{lax} P. D. Lax, {\it A stability theorem for solutions of
    abstract differential
 equations, and its application to the study of the local behavior of solutions of elliptic equations}, Comm. Pure Appl. Math., 9 (1956), 747--766. 

\bibitem{lees-protter}
M. Lees, M. H. Protter, {\it Unique continuation for parabolic
  differential
 equations and inequalities}, Duke Math. J., 28 (1961), 369--382.


\bibitem{poon} C-C. Poon, {\it Unique continuation for parabolic
    equations},  Comm. Partial Differential Equations, 21 (1996),
  no. 3-4, 521--539.

\bibitem {SH} R. E. Showalter, {\it Hilbert space methods for partial
   differential equations}, Monographs and Studies in Mathematics,
 Vol. 1. Pitman, London-San Francisco, Calif.-Melbourne, 1977.

\bibitem{S} J. Simon, {\it Compact sets in the space $L^p(0, T;B)$},
 Ann. Mat. Pura Appl., (4) 146 (1987), 65--96.


\bibitem{vazquez_zuazua} J. L. Vazquez, E. Zuazua, {\it The Hardy
    inequality and the asymptotic behaviour of the heat equation with
    an inverse-square potential},  J. Funct. Anal.,  173 (2000), no. 1,
    103--153.

\bibitem{watson}
{\sc Watson, G. N.}, {\it A treatise on the theory of Bessel
functions}, 2d ed., Cambridge Univ. Press, London,
England, 1944.

\end{thebibliography}
\end{document}